\documentclass{amsart}
\usepackage{shortsalch} 
\usepackage{float}
\usepackage{subcaption}   
\usepackage{xcolor}

\usepackage{amsaddr}   
 
\makeatletter
\renewcommand{\email}[2][]{%
  \ifx\emails\@empty\relax\else{\g@addto@macro\emails{,\space}}\fi%
  \@ifnotempty{#1}{\g@addto@macro\emails{\textrm{(#1)}\space}}%
  \g@addto@macro\emails{#2}%
}
\makeatother

\usepackage{graphicx}
\graphicspath{{./EU-symbol.jpg}}

\title{Commuting unbounded homotopy limits with Morava K-theory}
\author{Gabriel Angelini-Knoll$^*$ \and Andrew Salch$^\dagger$}
\address{Universit\'e Sorbonne Paris Nord,
LAGA, 
93430, Villetaneuse, 
France.$^{*}$
}
\address{Department of Mathematics, Wayne State University\\ 
Detroit, Michigan, U.S.A.$^\dagger$}
\email{gak@math.fu-berlin.de,asalch@wayne.edu}

\usepackage{cleveref}

\begin{document}
\begin{abstract}
This paper provides conditions for Morava $K$-theory to commute with certain homotopy limits. These conditions extend previous work on this question by allowing for homotopy limits of sequences of spectra that are not uniformly bounded below. 
As an application, we prove the $K(n)$-local triviality (for sufficiently large $n$) of the algebraic $K$-theory of algebras over truncated Brown--Peterson spectra, building on work of Bruner--Rognes and extending a classical theorem of Mitchell on 
$K(n)$-local triviality of the algebraic K-theory spectrum of the integers for large enough $n$.
\end{abstract} 
\maketitle
 
\tableofcontents 
  
\section{Introduction}   
Given a generalized homology theory $E_*$ and a sequence 
\[\dots \rightarrow X_{2} \rightarrow X_{1} \rightarrow X_0\] 
of spectra, one often needs to know whether there is an isomorphism
\[ \underset{i}{\lim\thinspace} E_*(X_i)\cong E_*(\underset{i}{\holim}X_i).\]
This cannot be true in full generality. For example, the limit of the sequence 
\[ \dots \rightarrow S/p^2\rightarrow S/p\]
is the $p$-complete sphere $S_p^{\wedge}$. Consequently,
\[H_*(\underset{i}{\holim\thinspace} S/p^i;\mathbb{Q})\cong H_*(S_p^{\wedge};\mathbb{Q})\cong \mathbb{Q}_p.\] 
On the other hand, $H_*(S/p^i;\mathbb{Q})\cong 0$ for $i\ge 1$, and therefore 
\[\underset{i}{\lim\thinspace} H_*(S/p^i;\mathbb{Q})\cong 0.\] 
Additionally, $R^1\underset{i}{\lim\thinspace} H_*(S/p^i;\mathbb{Q})$ vanishes, so we cannot even recover the groups $H_*(\underset{i}{\holim\thinspace}S/p^i;\mathbb{Q})$
from a ``Milnor sequence.'' This example can be generalized to Morava K-theory $K(n)_*$ for $n>0$ by letting $V$ be a type $n$, finite spectrum
with $v_n$-power self map
 $v\colon \Sigma^dV\to V$ and considering the limit of the sequence 
\[ \dots \rightarrow \cof(v^2) \rightarrow \cof(v)\]
where $\cof(v^i)$ denotes the cofiber of the iterated composite $v^i\colon \Sigma^{di}V\to V$. 

This motivates the question: what conditions on $E_*$ and the sequence
\[ \dots \rightarrow X_2 \rightarrow X_1 \rightarrow X_0\] 
allow us to commute the homotopy limit with $E_*$? There are known results along these lines, most famously a commonly-used result of Adams from \cite{MR1324104}, but the usual hypotheses are that 
\begin{itemize}
\item the spectra $X_i$ are uniformly bounded below, and 
\item the homology theory $E_*$ is connective. \end{itemize}
In this paper, we remove each of these assumptions, under some reasonable additional hypotheses. Our particular focus is on the case where $E_*$ is a Morava $K$-theory $K(n)_*$. 

There is previous unpublished work of Sadofsky \cite{Sad1,Sad2} on this question in the case when each $X_{i}$ is $K(n)$-local. Given a seqeunce 
\[ \dots \rightarrow X_2 \rightarrow X_1 \rightarrow X_0\] 
of $K(n)$-local spectra, Sadofsky constructs a spectral sequence computing the groups $K(n)_{*}(\lim_{i}X_{i})$ whose $E_{2}$-term is the right derived functors of lim in the category of $K(n)_{*}K(n)$-comodules (cf. \cite[Theorem A.0.1]{Pet20}). One may therefore view the higher right derived functors in $K(n)_{*}K(n)$-comodules as an obstruction to commuting the limit with Morava K-theory. 
Our results are, in a sense, orthogonal to the work of Sadofsky since we consider cases where the Sadofsky spectral sequence does not provide useful information because it does not convergence. 

This paper is written with a view towards filtered spectra that arise when studying topological periodic cyclic homology, 
\[TP(R):=THH(R)^{t\mathbb{T}}.\]
In particular, the Greenlees filtration \cite{MR908451} on topological periodic cyclic homology is not uniformly bounded below. Nevertheless, these filtered spectra often have nice enough homological properties to apply the main result of this paper. 

Following the red-shift program of Ausoni--Rognes \cite{AR08}, we are most interested in the chromatic complexity of topological periodic cyclic homology and related invariants. Therefore, a generalized homology theory of primary interest is Morava K-theory $K(n)_*$. Calculating Morava $K$-theory of topological periodic cyclic homology using the Greenlees filtration requires that one be able to commute a non-bounded-below generalized homology theory (Morava K-theory) with a non-uniformly-bounded-below homotopy limit, so existing results on generalized homology of limits, like Adams' theorem from \cite{MR1324104} reproduced as Theorem \ref{adams thm on holims} below, do not suffice. 

Our main result may then be summarized as follows.
\begin{theorem}[Theorem \ref{thm on holim of margolis-acyclics with lim1 vanishing}]\label{main thm}
Fix an integer $M$ and a prime $p$. Let 
\[\dots \rightarrow X_{2} \rightarrow X_{1} \rightarrow X_0  \rightarrow  X_{-1}  \rightarrow  \dots \] 
be a sequence of spectra satisfying:
\begin{enumerate}
\item each $X_i$ is bounded below and $p$-complete,
\item  for each $i$, the graded $\mathbb{F}_{p}$-vector space $H_*(X_i;\mathbb{F}_p)$ is finite type,
\item  the sequence of graded $\mathbb{F}_p$-vector spaces $\{H(H_*(X_i ; \mathbb{F}_p),Q_{n})\}$ is pro-isomorphic to zero, 
\item  there exists an integer $L$ such that for each $s\ge L$ the groups $\pi_{s}(\lim_{i}X_{i})/p$ are finite, and
\item  for each $i$, the $A_*$-comodule primitives of $H_*(X_{i};\mathbb{F}_p)$ are trivial in grading degrees $\ge N$.
\end{enumerate}
Then $\lim_{i}X_{i}$ is $K(n)$-acyclic. 
\end{theorem}

As the main application in the present paper, we prove a higher chromatic height analogue of Mitchell's theorem for algebraic K-theory of truncated Brown--Peterson spectra, building on work of Bruner--Rognes \cite{BR05}. In particular, Mitchell proves in \cite{Mit90} that 
$K(m)_*(K(\mathbb{Z}))$ vanishes
for $m\ge 2$, and consequently the algebraic $K$-theory of every $H\mathbb{Z}$-algebra is $K(m)$-acyclic. First, we fix conventions.
\begin{definition}[{cf. \cite[Definition 4.1]{LN14}}]\label{forms} Let $n$ be an integer, let $m\in \{ 1, 2, \dots \} \cup \{ \infty\}$, and let $p$ be a prime number. By a {\em $p$-primary $E_m$ form of $BP\langle n\rangle$} we mean a $p$-local $E_m$ ring spectrum $R$ equipped with a complex orientation $MU_{(p)}\rightarrow R$ such that the composite map
\[ \mathbb{Z}_{(p)}[v_1, \dots ,v_n] \hookrightarrow \pi_{*}BP \hookrightarrow \pi_*(MU_{(p)}) \rightarrow \pi_*R\]
is an isomorphism.
\end{definition}
We also note that, for every $m$, any two $p$-primary $E_m$ forms of $BP\langle n\rangle$ become homotopy-equivalent (as spectra) after $p$-completion. This is a consequence of the main theorem of \cite{AL17}. 

Our proof builds on a result of Bruner and Rognes \cite[Prop. 6.1]{BR05}, in which it is proven that the $E_{\infty}$-term of the homological homotopy fixed-point spectral sequence 
\begin{align}
\label{HFPSS first mention}H^*(\mathbb{T},H_*(THH(R);\mathbb{F}_p)) &\Rightarrow H^c_*(TC^{-}(R);\mathbb{F}_p)\end{align}
is isomorphic to $\left( P(t)\otimes M_1\otimes M_2 \right) \oplus T$ in the case $R$ is a $p$-primary $E_{\infty}$-form of $BP\langle n\rangle$. Here $P(t)$ and $M_1$ and $M_2$ and $T$ are certain comodules over the dual Steenrod algebra; see Section \ref{A higher height Mitchell theorem} for an explicit description of these comodules.

In this paper, we offer some applications of our main theorem to Morava $K$-theory of algebraic $K$-theory spectra of various forms of $BP\langle n\rangle$.

We do not really need the full strength of an $E_{\infty}$ form of $BP\langle n\rangle$: for our arguments, it is enough to have an $E_2$-form $R$ of $BP\langle n\rangle$ {\em such that the $E_{\infty}$-page of spectral sequence \eqref{HFPSS first mention} is isomorphic to that computed by Bruner and Rognes as a comodule over the $p$-primary dual Steenrod algebra}\footnote{We suspect, but not made serious efforts to prove, that a significantly weaker form of $BP\langle n\rangle$ suffices for this. See Remark \ref{conditional remark}, below, for further discussion.}. This leads us to the following assumption:
\begin{runningassumption}\label{assumption}
Let $B\langle n\rangle$ be a $p$-primary $E_{2}$-form of $BP\langle n\rangle$ such that the conclusion of \cite[Prop. 6.1]{BR05} holds. 
\end{runningassumption}
By the previous discusssion, we know that Running Assumption \ref{assumption} holds for and $p$-primary $E_{\infty}$-form of $BP\langle n\rangle$, which exist at all primes when $n=1$ by \cite{MR2182697}, when $p=n=2$ by \cite{LN14} and when $n=2$ and $p=3$ by \cite{MR2671186}. 
Let $K(R)$ denote the algebraic K-theory spectrum associated to an $E_{1}$ ring spectrum $R$. 
\begin{theorem}[Theorem \ref{main application}]\label{main application thm}
Fix a prime $p\ge 3$ and let $B\langle n\rangle$ be a $p$-primary $E_{2}$ form of $BP\langle n\rangle$ satisfying Running Assumption \ref{assumption}. For example, let  $B\langle n\rangle$ be any $p$-primary $E_{\infty}$ form of $BP\langle n\rangle$. Let $m$ be an integer, $m\geq n+2$.
Then there is a weak equivalence
\[L_{K(m)}K (B\langle n \rangle )\simeq 0 .\]
Consequently, the algebraic $K$-theory spectrum $K(A)$ is $K(m)$-acyclic for every $B\langle n\rangle$-algebra $A$.
\end{theorem} 

This result recovers the vanishing of $K(m)_*(K(\mathbb{Z}_{(p)}))$ for $m\ge 2$ by Mitchell \cite{Mit90}, since $H\mathbb{Z}_{(p)}$ is an $E_{\infty}$ form of $BP\langle 0\rangle$. Our result also proves that $K(m)_*(K(\ell))=0$ for $m\ge 3$ and $p\ge 3$, since the Adams summand $\ell$ is an $p$-primary $E_{\infty}$ form of $BP\langle n\rangle$ by \cite{MR2182697}. This recovers the vanishing of $K(m)_*(K(\ell))$ for $m\ge 3$ and $p\ge 5$ proven by Ausoni--Rognes \cite{AR02}, although our proof uses entirely different methods. At $p=3$,  there is an $E_{\infty}$-ring spectrum $\text{taf}^{D}$, which is an $E_{\infty}$ form of $BP\langle 2\rangle$ constructed by Hill--Lawson \cite{MR2671186} using a Shimura variety $\mathcal{X}^{D}$ associated to a quaternion algebra of discriminant $14$ . Our result therefore also proves that $K(m)_*K(\text{taf}^{D})$ vanishes for $m\ge 4$. 

We say a ring spectrum $R$ has height $n$ if $K(m)_*R=0$ for all $m>n$ and $K(n)_*R\ne 0$. Our result proves, in particular, that $K(BP\langle 2\rangle)$ has height $\le 3$. To prove that $K(BP\langle 2\rangle)$ has height exactly $3$, one needs to prove the nonvanishing result $K(3)_*K(BP\langle 2\rangle)\ne 0$. Since the first draft of this paper, Hahn--Wilson \cite{HW22} have shown that in fact 
$K(n+1)_*K(BP\langle n\rangle )\ne 0$
for any $p$-primary $E_3$-form of $BP\langle n\rangle$. In fact, Hahn--Wilson \cite{HW22} prove the stronger statement that $K(BP\langle n\rangle )$ has fp-type $n+1$ in the sense of \cite[p.5]{MR99}.

After work of Mitchell in 1990 \cite{Mit90} and work of Ausoni--Rognes in 2002  \cite{AR02}, there has been renewed interest and progress on questions of this nature. In particular, after the first draft of this paper was posted in pre-print form, work of Clausen--Mathew--Naumann--Noel \cite{CMNN24} and Land--Mathew--Meier--Tamme \cite{LMMT24} proved that if $R$ is an $E_1$-ring spectrum and $T(n-1)_*R=0$ and $T(n)_*R=0$, then $T(n)_*K(R)=0$, where $T(n)$ is the $v_n$-telescope $v_n^{-1}V$ of a type $n$ finite complex $V$.
We note that our Theorem \ref{main application} is proven using entirely different techniques to these other authors and a feature of our approach is that it applies to topological periodic cyclic homology. We therefore regard our methods as complementary to those of \cite{CMNN24,LMMT24}. 

Our work in this paper was motivated initially by a conjecture in an early draft of a paper by the first author and Quigley \cite{AKQ25}. Let $y(n)$ be the Thom spectrum of the $1$-fold loop map 
\[ \Omega J_{2^n-1}(S^2)\to \Omega J(S^2)\simeq \Omega^2S^3\to BGL_1S\]
where $\Omega^2S^3\to BGL_1S$ is the unique (up to homotopy) $2$-fold loop map, with target $BGL_1S$ is (a model for) the classifying space of stable spherical bundles, and $J(S^2)$ is the James construction \cite{Jam55} equipped with its usual filtration 
\[ J_{2^n-1}(S^2)\hookrightarrow J(S^2)\]
for all $n\ge 0$.
Here we write $TP(y(n))[k]$ for the Greenlees filtration of topological periodic cyclic homology of $y(n)$ (see  \cite{MR908451} for details about this filtration).  

In \cite{AKQ25}, the first author and Quigley conjectured that the pro-vanishing of Margolis homology of $H_{*}(TP(R)[k];\mathbb{F}_{p})$ should be sufficient to show that $K(n)_{*}TP(R)$ vanishes, in the case of $R=y(n)$. This conjecture is resolved by Theorem \ref{main thm} together with calculations of the first author and Quigley which verify that the remaining hypotheses of Theorem \ref{main thm} are satisfied. 
This is used by the first author and Quigley to prove that $L_{K(m)}y(n)=0$ implies $L_{K(m+1)}TP(y(n))=0$ for $0\le m<n$ and each $n\ge 0$ in \cite{AKQ25}. 

\subsection{Conventions}\label{conventions}
Our conventions are all relatively standard, but we prefer to state them explicitly to avoid any possible confusion.

When $\kappa$ is a cardinal number and $A$ an object in some category, we will write $A^{\kappa}$ for the $\kappa$-fold categorical product and $A^{\coprod\kappa}$ for the $\kappa$-fold categorical coproduct. If our category is additive and $\kappa$ is finite, then we write $A^{\oplus\kappa}$ for the $\kappa$-fold categorical biproduct of $A$ with itself.

Let $R$ be a commutative ring. 
In this paper, the term ``finite type'' is used in the sense common in algebraic topology: a graded $R$-module $V$ is {\em finite type} if, for each integer $n$, the grading degree $n$ summand $V^n$ is a finitely generated $R$-module. In particular, if $R$ is a field, then we say $V$ is finite type if, for each integer $n$, the grading degree $n$ summand $V^n$ is a finite dimensional vector space. 

We will write $P(x_1,\dots ,x_n)$ for a polynomial algebra over $\mathbb{F}_{p}$ with generators $x_1,\dots ,x_n$. We write $E(x_1,\dots ,x_n)$ for an exterior algebra over $\mathbb{F}_p$ with generators $x_1,\dots ,x_n$. We write $P_k(x_1,\dots,x_n)$ for the truncated polynomial algebra over $\mathbb{F}_p$ with generators $x_1,\dots ,x_n$ and relations $x_1^k, ... ,x_n^k$. Note that $P_k(x)$ is often also denoted $\mathbb{F}_p[x]/x^k$ in the literature.

Suppose we have a ring $R$, a non-zero-divisor $r\in R$, and a left $R$-module $M$.
\begin{itemize}
\item An element $m\in M$ is {\em $r$-power-torsion} if there exists an integer $n$ such that $r^nm = 0$.
\item The module $M$ is {\em $r$-power-torsion} if, for each $m\in M$, there exists some integer $n$ such that $r^nm = 0$.
\item An element $m\in M$ is {\em simple $r$-torsion} if $rm=0$.
\item The module $M$ is {\em simple $r$-torsion} if $rm = 0$ for all $m\in M$.
\end{itemize}

Given categories $\cC$ and $\cD$, we will write $\cD^{\cC}$ for the category of functors from $\cC$ to $\cD$. By a \emph{sequence} of objects in a category $\mathcal{C}$, we mean an object in $\mathcal{C}^{(\mathbb{Z}^{\op})}$, where $\mathbb{Z}$ is regarded as a partially-ordered set and hence as a small category. Having adopted this convention, whenever we write $\lim_{k}$, it always indicates a limit as $k\rightarrow\infty$. When we write $\colim_{k}$, it always indicates a limit as $k\rightarrow -\infty$.
We use the notation $\{X_i\}$ as shorthand for a sequence of objects $\dots \rightarrow X_1 \rightarrow X_0 \rightarrow X_{-1}\rightarrow \dots$ in $\mathcal{C}$. 

Let $p$ be an odd prime.  We write $A$ for the $p$-primary dual Steenrod algebra and $A_{*}$ for its dual. By Milnor \cite{MR0099653}, there is an isomorphism
\[ A_*\cong P( \bxi_i : i\ge 1 ) \otimes E(\btau_i: i \ge 0)\]
where $\bxi_i$ and $\btau_i$ are the conjugates of Milnor's generators $\xi_i$ and $\tau_i$. That is, if we write $\chi$ for the antipode map of the Hopf algebra $A_*$, then $\bxi_i=\chi(\xi_i)$ and $\btau_i=\chi(\tau_i)$. The coproduct is defined by the formulas
\begin{align} 
\label{xi coprod}\Delta( \bxi_k) = \sum_{i+j=k} \bxi_i \otimes \bxi_j^{p^i} &\text{ and }  \\
\label{tau coprod} \Delta( \btau_k) = 1\otimes \btau_k + \sum_{i+j=k} \btau_i\otimes \bxi_j^{p^i} &
\end{align}
where $\bxi_0=1$ by convention. Given $A_{*}$-comodules $M$ and $n$, we write  $\Hom_{A_{*}}(M,N)$ for the $\mathbb{F}_{p}$-vector space of $A_{*}$-comodule maps from $M$ to $N$ and $\Ext_{A_{*}}^{*.*}(\mathbb{F}_{p},N)$ the right derived functors of the functor $\Hom_{A_{*}}(\mathbb{F}_{p},-)$.

\subsection{Organization}
In Section \ref{Hm condition}, we give sufficient conditions for the canonical map $X\to X^{<N}$ to induce an injection on mod $p$ homology. 
In Section \ref{sec: acyclic limits}, we prove the main theorem.
In Section \ref{application}, we give our 
main application, which is a proof of a higher chromatic height analogue of Mitchell's theorem. 
We also provide an appendix 
(Appendix \ref{appendix on basics of margolis homology}), 
including results on Margolis homology.  

\subsection{Acknowledgements}
The first author would like to thank J.D. Quigley for discussions related to this project and Bj{\o}rn Dundas and Eric Peterson for expressing their interest in the project. The authors would also like to thank an anonymous referee as well as John Rognes and Dexter Chua for careful readings of the paper that lead to significant improvements. This project has received funding from the European Union's Horizon 2020 research and innovation programme under the Marie Sk\l{}odowska-Curie grant agreement No 1010342555.
\thinspace \includegraphics[scale=0.15]{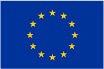} 

\section{Morava K-theory of homotopy limits}

Let $X$ be a spectrum, and let $N$ be an integer. Let $X^{<N}$ denote the spectrum obtained by attaching cells to $X$ to kill the homotopy groups in degrees greater than or equal to $N$. In \cref{Kn-acyclicity of products}, we will find ourselves needing an answer to the following natural question: when does the natural map
\[ X \to X^{<N} \]
induce a {\em one-to-one} map in mod $p$ homology? Consequently our first task, in Section \ref{Hm condition}, is to answer that question.

\subsection{When does killing homotopy induce an injection in homology?}\label{Hm condition}
Let $P(M)=\Hom_{A_*}(\mathbb{F}_p,M)$ denote the graded $\mathbb{F}_p$-vector space of comodule primitives in a graded left $A_*$-comodule $M$, and let $Q(Z)=\mathbb{F}_p\otimes_A Z$
denote the $A$-module indecomposables of $Z$.

\begin{definition}\label{def of condition H(m)}
Let $p$ be a prime number and let $N$ be an integer. We will say $X$ {\em satisfies condition $H(N)$} if $X$ is bounded below, 
and the $A_*$-comodule primitives of $H_*(X;\mathbb{F}_p)$ are trivial in grading degrees $\ge N$. 
\end{definition}
We suppress the prime $p$ from the notation $H(N)$, because it will always be clear from the context. 
\begin{prop}\label{injectivity on homology prop}
Let $p$ be a prime number, let $N$ be an integer, and let $X$ satisfy condition $H(N)$.
Let $X^{<N}$ be $X$ with cells attached to kill all the homotopy groups of $X$ in degrees $\geq N$. Then the map $H_*(X; \mathbb{F}_p) \rightarrow H_*(X^{<N};\mathbb{F}_p)$, induced by the canonical map $X\rightarrow X^{<N}$, is injective. 
\end{prop}

\begin{proof}
It is classical (e.g. see Proposition 3.9 of \cite{MR0174052}) that a map $M\to M^{\prime}$ of {\em coalgebras} is injective if and only if it is injective on the primitives. We need to show that the same is true for a map $M\to M^{\prime}$ of bounded-below {\em comodules} instead. Dually, we want to know that a map of graded $A$-modules
$f\colon \thinspace Z\to Z^{\prime} $
is surjective if $Z^{\prime}$ is bounded below and the induced map on indecomposables
$ \mathbb{F}_p\otimes_{A}f \colon \thinspace Q(Z)\to Q(Z^{\prime}) $
is surjective. 
This is elementary: every homogeneous element $x$ of $Z^{\prime}$ is an $A$-linear combination $\sum_{i=1}^n a_iq_i$ of homogeneous elements $q_i\in Q(Z^{\prime})$. Lift each $q_i$ to an element $\tilde{q}_i\in Q(Z)$, and observe that $f(\sum_{i=1}^na_i\tilde{q}_i) = x$.
\end{proof}

\begin{example}
A simple example where the hypotheses, and therefore the result, holds, is when $X=S^0$ and $N=1$, so that $\mathbb{F}_p\hookrightarrow (A//E(0))_*$ is an inclusion. 
A simple example where the hypotheses are {\em not} satisfied is the case where $X=S^0$ and $N=0$, where of course $X^{<N}$ is contractible and the result cannot hold. 
\end{example}

\subsection{Recollections}
Recall that, for each prime number $p$ and positive integer $n$, we have the homotopy fiber sequence
\[ \Sigma^{2(p^n-1)} k(n) \rightarrow k(n) \rightarrow H\mathbb{F}_p, \]
and the composite map
\begin{align}\label{Qn map of spectra} 
	H\mathbb{F}_p \rightarrow \Sigma^{2p^n-1}k(n) \rightarrow \Sigma^{2p^n-1}H\mathbb{F}_p
\end{align}
is the cohomology operation $Q_n$, which satisfies $Q_n^2 = 0$. 
This implies a useful relationship between Morava $K$-theory and Margolis homology of $E(Q_n)$-modules, which we briefly summarize in an appendix to this paper, Appendix \ref{appendix on basics of margolis homology}.

\begin{definition}
We say that an $E(Q_n)$-module is {\em $Q_n$-acyclic} if the $Q_n$-Margolis homology $H(M;Q_n)$ vanishes. We say that a morphism of $E(Q_n)$-modules is a {\em $Q_n$-equivalence} if it induces an isomorphism in $Q_n$-Margolis homology.
\end{definition}

We recall the following useful result of Adams \cite{MR1324104}.
\begin{theorem}[Theorem III.15.2 \cite{MR1324104}]\label{adams thm on holims}
Suppose that $R$ is a subring of $\mathbb{Q}$, $E$ is a bounded-below spectrum such that $H_r(E;R)$ is a finitely generated $R$-module for all $r$, and $\{ X_i\}_{i\in I}$ is a set of spectra such that $\pi_r(X_i)$ is an $R$-module for all $r$. Suppose that there exists a uniform lower bound for $\pi_*(X_i)$, i.e., there exists an integer $N$ such that $\pi_n(X_i) \cong 0$ for all $n<N$. Then the canonical map of spectra
\begin{align}\label{comparison map 3049} E\smash \prod_{i\in I} X_i \rightarrow \prod_{i\in I} \left( E\smash X_i\right) \end{align}
is a weak equivalence.
\end{theorem}

The following corollary is a straightforward consequence of Theorem \ref{adams thm on holims}.
\begin{corollary}\label{adams cor}
Let $\{X_i\}$ be a sequence of spectra. Suppose that there exists a uniform lower bound on $\pi_*(X_i)$. Then the canonical maps of spectra
\begin{align*} H\mathbb{F}_p\smash \underset{i}{\holim}\thinspace X_i &\rightarrow \underset{i}{\holim}\thinspace \left( H\mathbb{F}_p\smash X_i\right) \mbox{\ \ \ and} \\
 k(n)\smash \underset{i}{\holim}\thinspace X_i &\rightarrow \underset{i}{\holim}\thinspace  \left( k(n) \smash X_i\right) \end{align*}
are weak equivalences.
\end{corollary}

\subsection{$K(n)$-acyclicity of products.}
\label{Kn-acyclicity of products}
\begin{definition}
Given a spectrum $X$, by the {\em Whitehead filtration of $X$} we mean the functor $\Wh(X): \mathbb{Z}^{\op} \rightarrow \text{Sp}$ given by letting $\Wh^{n}(X) = \Wh(X)(n)$ be the fiber $X^{\ge n}$ of the canonical map $X\to X^{< n}$. The map $\Wh^{n+1}(X)\rightarrow \Wh^{n}(X)$ is the natural map $X^{\geq n+1} \rightarrow X^{\geq n}$.
\end{definition}
One agreeable property of Whitehead filtrations is that they are compatible with products:
\begin{lemma}\label{whitehead filt and products}
Let $I$ be a set, and for each $i\in I$, let $X_i$ be a spectrum. Then the natural map $\Wh^{n}(\prod_{i\in I} X_i)\rightarrow\prod_{i\in I} \Wh^{n}(X_i)$ is a weak equivalence for each integer $n$. Moreover, these weak equivalences are compatible with the maps in the Whitehead filtrations, so that 
\[ \Wh\left(\prod_{i\in I}X_i\right) \rightarrow \prod_{i\in I} \Wh(X_i) \]
is a levelwise weak equivalence of sequences in $\SHC$.
\end{lemma}
\begin{proof}
Routine consequence of $\pi_*$ commuting with arbitrary products.
\end{proof}

\begin{lemma}\label{ss existence lemma}
Let $p$ be a prime number, let $n$ be a positive integer, let $I$ be a set, and for each $i\in I$, let $X_i$ be a spectrum. We have a strongly convergent spectral sequence
\begin{align}
\label{ss 1000000} E^1_{s,t} \cong \prod_{i\in I} k(n)_t\left( \Sigma^s H\pi_s(X_i)\right) &\Rightarrow k(n)_t\left( \prod_{i\in I} X_i\right) \\
\label{ss 1000001} d^r: E^r_{s,t} &\rightarrow E^r_{s+r,t-1}
\end{align}
given by the exact couple arising from applying connective Morava $K$-theory $k(n)$ to the Whitehead tower of $\prod_{i\in I} X_i$.
The action of $v_n\in k(n)_{2(p^n-1)}$ on the abutment increases filtration by one, so that its action in the spectral sequence is zero. 
\end{lemma}
\begin{proof}
Throughout, we use the terminology and methods of Boardman's paper \cite{MR1718076}.
Consider the spectral sequence of the unrolled exact couple obtained by applying $k(n)_*$ to the Whitehead filtration of $\prod_{i\in I} X_i$.
For each integer $t$, the $k(n)$-homology group $k(n)_t\left( \Wh^{s}\left(\prod_{i\in I} X_i\right)\right)$ vanishes for all $s>t$, by the Hurewicz theorem. Consequently 
\[ \lim_s k(n)_t\left( \Wh^{s}\left(\prod_{i\in I} X_i\right) \right) \text{ and } R^1\lim_sk(n)_t\left( \Wh^{s}\left(\prod_{i\in I} X_i\right) \right) \]
each vanish for all $t$. Consequently the spectral sequence converges conditionally to the colimit.  

The description of the $E^1$-page given in \eqref{ss 1000000} follows from Lemma \ref{whitehead filt and products}. The spectral sequence's $E^1$-page vanishes in bidegrees $(s,t)$ with $t<s$, so by \eqref{ss 1000001}, it is a half-plane\footnote{That is, the spectral sequence is half-plane after a suitable linear re-indexing of bidegrees to move the line $t-s=0$ to a coordinate axis. This re-indexing does not affect the fact that the spectral sequence has exiting differentials, in Boardman's sense.} spectral sequence with exiting differentials, convergent to the colimit. By Theorem 6.1 of \cite{MR1718076}, the spectral sequence must then be strongly convergent.

The spectral sequence converges to the colimit, which is 
\begin{align}
\nonumber \colim_s k(n)_*\left( \Wh^{s}\left( \prod_{i\in I} X_i\right)\right) 
\nonumber  &\cong k(n)_*\left( \hocolim_s \left( \Wh^{s}\left( \prod_{i\in I} X_i\right)\right)\right) \\
\nonumber  &\cong k(n)_*\left( \prod_{i\in I} X_i\right) .
\end{align}
Hence the spectral sequence converges to the $k(n)$-homology of the product $\prod_{i\in I} X_i$, as claimed.
\end{proof}

\begin{lemma}\label{kn injectivity lemma}
Let $p$ be a prime and let $n$ be a positive integer. Write $k(n)$ for the $n$th $p$-primary connective Morava $K$-theory spectrum. Let $M$ be an integer, and let $X$ be a bounded-below spectrum satisfying condition $H(M)$. Suppose that $k(n)_*(X)$ is simple $v_n$-torsion.
Then the canonical map $k(n)_*(X) \rightarrow k(n)_*(X^{<M})$ is injective.
\end{lemma}
\begin{proof}
Since $X$ is bounded-below, the smash product $k(n)\smash X$ is also bounded-below, and moreover has $p$-adically complete homotopy groups. Consequently Bousfield's convergence results in \cite{MR551009} yield that the $H\mathbb{F}_p$-nilpotent completion map $k(n)\smash X\rightarrow \left( k(n)\smash X\right)^{\wedge}_{H\mathbb{F}_p}$ is an equivalence.
Hence the Adams spectral sequence
\begin{align} \label{adams ss 00112}\Ext_{Comod(E(Q_n)^*)}^{s,t}\left(\mathbb{F}_p,H_*(X; \mathbb{F}_p)\right)
 &\Rightarrow \pi_{t-s}\left( \left( k(n)\smash X\right)^{\wedge}_{H\mathbb{F}_p}\right)
\end{align}
converges to the $k(n)$-homology groups of $X$. Since $k(n)_*(X)$ is simple $v_n$-torsion, the $E_{\infty}$-term of \eqref{adams ss 00112} is concentrated on the $s=0$-line. Consequently the Hurewicz map $k(n)_*(X) \rightarrow H_*(X;\mathbb{F}_p)$ is injective. We have a commutative square
\[\xymatrix{
 k(n)_*X \ar[r] \ar[d] & k(n)_*(X^{<M}) \ar[d] \\
 H_*(X; \mathbb{F}_p)\ar[r] & H_*(X^{<M}; \mathbb{F}_p)
}\]
whose left-hand vertical map is injective. The square's bottom horizontal map is injective since $X$ satisfies condition $H(M)$. Consequently the top vertical map is also injective, as claimed.
\end{proof}
 
\begin{theorem}\label{main thm prod case}
Let $p$ be a prime number, let $M$ be an integer, let $n$ be a positive integer, let $I$ be a set, and for each $i\in I$, let $X_i$ be a bounded-below\footnote{We emphasize that we are not assuming a {\em uniform} lower bound on the spectra $X_i$.} spectrum satisfying condition $H(M)$. Suppose that $k(n)_*(X_i)$ is simple $v_n$-torsion for all $i\in I$. Finally, suppose that the $\mathbb{F}_p$-vector space $\prod_{i\in I}k(n)_t(H\pi_sX_i)$ is finite-dimensional for each $t$ and for each $s\geq M$.
Then the product $\prod_i X_i$ is $K(n)$-acyclic.
\end{theorem}
\begin{proof}
Since each $X_i$ satisfies condition $H(M)$, for each $i\in I$ the map $\Wh^{M}(X_i) \rightarrow X_i$ is zero in $k(n)$-homology, by Lemma \ref{kn injectivity lemma}. Consequently the map
\begin{align*}
\prod_{i\in I} k(n)_*\left(\Wh^{M}(X_i)\right) 
 &\rightarrow k(n)_*\left( \prod_{i\in I} \left( \Wh^{q}(X_i) \right)\right),
\end{align*}
defined for all $q\leq M$, 
is trivial on the factor $k(n)_*(\Wh^{m}(X_i))$ corresponding to any element $i\in I$ such that $X_i$ is $q$-connective.
 For each integer $s$, the image of the map 
 \begin{equation}\label{map 3340495064d} 
 k(n)_t\left(\prod_{i\in I}\left(\Wh^{s}(X_i) \right)\right) \rightarrow \colim_q k(n)_t\left( \prod_{i\in I} \left (\Wh^{q}(X_i)\right)\right)
 \end{equation}
  is the $s$th filtration layer in the abutment of the spectral sequence \eqref{ss 1000000}.
If $s>M$, then the map \eqref{map 3340495064d} factors through the map 
\begin{equation}\label{map 334049506099} 
\prod_{i\in I}k(n)_t\left( \Wh^{M}(X_i)\right)\rightarrow\colim_{q}k(n)_t\left(\prod_{i\in I} \Wh^{q}(X_i)\right),
\end{equation} so we aim to show that the map \eqref{map 334049506099} is trivial, yielding $E^{\infty}_{s,t}=0$ for all $s>M$. 

A homogeneous element of $\prod_{i\in I} k(n)_*\left(\Wh^M(X_i)\right)$ of degree $t$ is given by specifying, for each $i\in I$, an element $x_i\in k(n)_t\left(\Wh^M(X_i)\right)$. In the sequence of abelian groups 
\begin{equation}\label{map 3340495063} \prod_{i\in I}k(n)_t\left( \Wh^{M}(X_i)\right)
 \rightarrow \prod_{i\in I}k(n)_t\left( \Wh^{M-1}(X_i)\right)
 \rightarrow \prod_{i\in I}k(n)_t\left( \Wh^{M-2}(X_i)\right)
 \rightarrow \dots ,\end{equation}
$x_j\in k(n)_t(\Wh^{M}(X_j))$ is trivialized as soon as we reach $\prod_{i\in I}k(n)_t\left( \Wh^{u}(X_i)\right)$, where $u$ is the connectivity of $X_j$. 

Consequently, for each element $x$ of $\prod_{i\in I}k(n)_*\left( \Wh^{M}(X_i)\right)$, the image of $x$ in the colimit $\colim_{s}\prod_{i\in I}k(n)_t\left( \Wh^{s}(X_i)\right)$ is zero upon projection to the quotient 
\[\colim_{s}\prod_{i\in I_u}k(n)_t\left( \Wh^{s}(X_i) \right),\] 
where $I_u$ is the subset of $I$ consisting of those $i\in I$ such that the spectrum $X_i$ has connectivity precisely equal to $u$.
For each integer $u$, we have a spectral sequence
\begin{align}
\label{ss 1000006} {}^uE^1_{s,t} \cong \prod_{i\in I_u} k(n)_t\left( \Sigma^s H\pi_s(X_i)\right) &\Rightarrow k(n)_t\left( \prod_{i\in I_u} X_i\right) \\
\nonumber d^r: {}^uE^r_{s,t} &\rightarrow {}^uE^r_{s+r,t-1}
\end{align}
obtained by applying $k(n)_*$ to the product of the Whitehead filtrations of the spectra $X_i$ with $i\in I_u$. In spectral sequence \eqref{ss 1000006},
the bidegrees ${}^uE^1_{s,t}$ with $s<u$ are trivial. Consequently there is an upper bound on the lengths of nonzero differentials in spectral sequence \eqref{ss 1000006} which can hit bidegree ${}^uE^r_{s,t}$. 

The unrolled exact couple which yielded the original spectral sequence \eqref{ss 1000000} is the product, over all integers $u$, of the unrolled exact couple which yields the spectral sequence ${}^uE^*_{*,*}$ of \eqref{ss 1000006}. Since each given element $x$ of 
$\prod_{i\in I}k(n)_*\left( \Wh^{M}(X_i)\right)$
maps to zero in 
$\colim_{s}\prod_{i\in I_u}k(n)_t\left( \Wh^{s}(X_i)\right),$
the image of $x$ in $\prod_{i\in I_u}k(n)_*\left( \Wh^{M}(X_i)\right)$ must represent a class in the spectral sequence ${}^uE^*_{*,*}$ of \eqref{ss 1000006} which is hit by a differential. There is an upper bound, depending on $u$, on the length of that differential; this is a consequence of a lower vanishing line in the ${}^uE^1$-page of \eqref{ss 1000006}, simply due to the $u$-connectivity of the spectra whose $k(n)$-homology yields that ${}^uE^1$-page. Consequently, if we regard $x\in \prod_{i\in I}k(n)_*\left( X_i^{\geq M}\right)$ as an $I$-tuple $(x_i)_{i\in I}$, there is some finite page of spectral sequence ${}^uE^*_{*,*}$ by which all the components $x_i$ with $i\in I_u$ have been wiped out by differentials.

Hence, if $s\geq M$, the only chance for $x$ to represent a nonzero element in the abutment $\colim_{q}\prod_{i\in I}k(n)_t\left( X_i^{\geq q}\right)$ of spectral sequence \eqref{ss 1000000} is if $x$ represents an element in $E^1_{*,*}$ in a bidegree which is hit by arbitrarily long differentials (cf. Remark \ref{rem: explanation}). But whatever element of $E^1_{*,*}$ is represented by $x$, that element sits above the $s=M$-line, and consequently lives in a bidegree $E^1_{s,t}\cong \prod_{i\in I} k(n)_t\left( \Sigma^s H\pi_s(X_i)\right)$ which is, by assumption, finitely generated. Consequently only finitely many lengths of differentials can have nonzero image in that bidegree. Hence $x$ indeed must map to zero in the abutment of \eqref{ss 1000000}.
 
This yields the desired triviality of $E^{\infty}_{s,t}$ in \eqref{ss 1000000} for all $s>M$. Drawn with the Adams convention, its $E^1$-page also vanishes above the line $s=t$. In the following diagrams, the vanishing regions are colored green:
\begin{figure}[H]
\centering
\begin{subfigure}{.45\textwidth}
\centering
\begin{tikzpicture}[scale=1.7]
\clip (-0.2,-0.2) rectangle (1.6,1.6);
\fill[green!50!white] (-0.2,-0.2) -- (1.6,1.6) -- (-0.2,1.6) -- cycle;
\draw (-0.1,1.4) node{$s$};
\draw (1.4,-0.1) node{$t$};
\draw (-1.5,0) -- (1.8,0);
\draw (0,-1.5) -- (0,1.8);
\end{tikzpicture}
\caption{Vanishing region at $E^1$.}
\end{subfigure}
\begin{subfigure}{.45\textwidth}
\centering
\begin{tikzpicture}[scale=1.7]
\clip (-0.2,-0.2) rectangle (1.6,1.6);
\fill[green!50!white] (-0.2,-0.2) -- (1.6,1.6) -- (-0.2,1.6) -- cycle;
\fill[green!50!white] (-0.2,1.0) -- (1.6,1.0) -- (1.6,1.6) -- (-0.2,1.6) -- cycle;
\draw (-0.1,1.4) node{$s$};
\draw (1.4,-0.1) node{$t$};
\draw (-1.5,0) -- (1.8,0);
\draw (0,-1.5) -- (0,1.8);
\end{tikzpicture}
\caption{Vanishing region at $E^{\infty}$.}
\end{subfigure}
\end{figure}
The action of $v_n$ on the spectral sequence increases stem (i.e., position along the horizontal axis) by $2(p^n-1)$, and increases filtration (i.e., position along the vertical axis) by at least one in the $E^{\infty}$-page of spectral sequence \eqref{ss 1000000}, by the claim about the $v_n$-action in Lemma \ref{ss existence lemma}. Of course, there is also the possibility of longer filtration jumps, so that the action of $v_n$ on the abutment $k(n)_*\left(\prod_{i\in I} X_i\right)$ raises filtration by more than one. Nevertheless, the vanishing line in the $E^{\infty}$-page establishes that, starting in any given bidegree, there is an upper bound on the integers $j$ such that $v_n^j$ times elements in that bidegree can be nonzero in $k(n)_*\left(\prod_{i\in I} X_i\right)$, even taking into account possible filtration jumps. Consequently every homogeneous element of $k(n)_*\left(\prod_{i\in I} X_i\right)$ is $v_n$-power-torsion. Hence $v_n^{-1}k(n)_*\left(\prod_{i\in I} X_i\right) \cong K(n)_*\left(\prod_{i\in I} X_i\right)$ is trivial, as claimed.
\end{proof}

\begin{remark}\label{rem: explanation}
In the sequence 
\begin{equation}\label{seq 23004} \dots \rightarrow \prod_{i\in I}k(n)_t(\Wh^{M}(X_i)) \rightarrow  \prod_{i\in I}k(n)_t(\Wh^{M-1}(X_i)) \rightarrow\dots , 
\end{equation}
given an $I$-tuple $(x_i)_{i\in I}\in \prod_{i\in I}k(n)_t(\Wh^{M}(X_i))$, we know that all the components $x_i$ with $i\in I_u$ map to zero at the stage $\prod_{i\in I}k(n)_t(\Wh^{q}(X_i))$, so every individual $x_i$ eventually maps to zero. However, we must be careful about the possibility that $(x_i)_{i\in I}$ still does not map to zero in the colimit of \eqref{seq 23004}. This is the reason why we need the assumption that the $\mathbb{F}_p$-vector space $\prod_{i\in I}k(n)_t(H\pi_sX_i)$ is finite-dimensional for each $t$ and for each $s\geq M$. For a simple example of what could go wrong, consider the image of the element $(1,1,\dots)\in \prod_{n\geq 0} \mathbb{Z}$ in the colimit of the sequence of abelian groups
\[ \prod_{n\geq 1} \mathbb{Z} \stackrel{f_1}{\longrightarrow} \prod_{n\geq 1} \mathbb{Z}\stackrel{f_2}{\longrightarrow} \dots\]
in which $f_n$ sends the first $n$ factors of $\prod_{n\geq 1} \mathbb{Z}$ to zero, and is the identity on the remaining components.

\end{remark}

\subsection{$K(n)$-acyclicity of sequential homotopy limits.}\label{sec: acyclic limits}

\begin{lemma}
Let $\{X_i\}$ be a sequence in the stable homotopy category. Then we have a commutative diagram of spectra
\begin{equation}\label{comm diag 7732}\xymatrix{
 \vdots \ar[d] &
  \vdots \ar[d] &
  \vdots \ar[d] \\
\Wh^{q+1} \left( \holim_i X_i\right) \ar[r] \ar[d] &
  \holim_i \left( \Wh^{q+1}(X_i)\right) \ar[r] \ar[d] &
  \Sigma^{q} H\left(R^1\lim_i \pi_{q+1}X_i\right) \ar[d] \\
 \Wh^{q}\left( \holim_i X_i\right) \ar[r] \ar[d] &
  \holim_i \left( \Wh^{q}(X_i)\right) \ar[r] \ar[d] &
  \Sigma^{q-1} H\left(R^1\lim_i \pi_qX_i\right) \ar[d] \\
 \Wh^{q-1}\left( \holim_i X_i\right) \ar[r] \ar[d] &
  \holim_i \left( \Wh^{q-1}(X_i)\right) \ar[r] \ar[d] &
  \Sigma^{q-2} H\left(R^1\lim_i \pi_{q-1}X_i\right) \ar[d] \\
 \vdots &
  \vdots &
  \vdots 
}\end{equation}
whose rows are each homotopy fiber sequences.
\end{lemma}
\begin{proof}
To get that each row in \eqref{comm diag 7732} is a fiber sequence, compare the Milnor sequence for the homotopy groups of $\holim_iX_i$ to the Milnor sequence for the homotopy groups of $\holim_i \Wh^{q}(X_i)$. 
The compatibility with the structure maps of the sequences of spectra follows from functoriality of $\Wh^q$ and the universality of the comparison map from $\Wh^q\circ \holim_i$ to $\holim_i\circ \Wh^q$.
\end{proof}

\begin{lemma}\label{holim hocolim swap}
Let $\{X_{i}\}$ be a sequence of morphisms of spectra. Then the natural map 
\[ \underset{i}{\holim} X_i \rightarrow \underset{q}{\hocolim}\ \underset{i}{\holim} \left( \Wh^{q} (X_i) \right)\]
is an equivalence.
\end{lemma}
\begin{proof}
Take the homotopy colimit of each column in \eqref{comm diag 7732}.
\end{proof}

\begin{prop}\label{lim ss prop}
Let $\{X_{i}\}$ be a sequence of morphisms of spectra. Suppose that $R^1\lim_i \pi_*(X_i)$ vanishes. Then there exists a conditionally convergent spectral sequence
\begin{align}
\label{ss f43209} E^1_{s,t} \cong k(n)_t\left( \Sigma^s H\lim_i\pi_sX_i\right)   &\Rightarrow k(n)_t\holim_i X_i \\
\nonumber d_r: E^r_{s,t} &\rightarrow E^r_{s+r,t-1}.
\end{align}
\end{prop}
\begin{proof}
For each $i\in\mathbb{N}$, we have the Whitehead filtration $\Wh(X_i)$, a sequence of spectra. Consider the homotopy limit $\holim_i \Wh(X_i)$ of the Whitehead filtrations, and regard $\holim_i \Wh(X_i)$ as a tower of spectra. Applying $k(n)_*$ to that tower yields an exact couple. The claimed spectral sequence \eqref{ss f43209} is the spectral sequence of that exact couple.

By a connectivity argument exactly like the one used in the proof of Lemma \ref{ss existence lemma}, for each fixed integer $t$ the group $k(n)_t\left( \holim_i\Wh^s(X_i)\right)$ is trivial for sufficiently large $s$, so the spectral sequence converges conditionally to the colimit. 

A priori, that colimit is $\colim_{q} k(n)_t\left( \holim_i\Wh^q(X_i)\right)$. We must show that this colimit agrees with the claimed abutment $k(n)_t\holim_i X_i$. This agreement follows from the isomorphisms:
\begin{align}
\nonumber \colim_{q} k(n)_t\left( \holim_i \Wh^{q}(X_i)\right)
  &\cong k(n)_t\left( \hocolim_{q} \holim_i \left( \Wh^{q}(X_i) \right )\right ) \\
\label{iso m21409}  &\cong k(n)_t(\holim_i X_i)
\end{align}
where isomorphism \eqref{iso m21409} is due to Lemma \ref{holim hocolim swap}. 

We also must show that the spectral sequence's input is as claimed. We have the fiber sequence
\begin{equation}\label{fib seq 34523} \holim_i (\Wh^{s+1}(X_i))\rightarrow \holim_i (\Wh^{s}(X_i))\rightarrow \holim_i \Sigma^{s}H\pi_{s}X_i,\end{equation}
and using the Milnor exact sequence, the right-hand term in \eqref{fib seq 34523} is easily seen to be a two-stage Postnikov system with the same homotopy groups as
\begin{equation}\label{wedge 3409} \Sigma^{s} H\lim_i \pi_{s}X_i \vee  \Sigma^{s-1}HR^1\lim_i \pi_{s}X_i.\end{equation}
The vanishing assumption on $R^1\lim{}_i\pi_*(X_i)$ now ensures that the $E^1$-term is as claimed.
\end{proof}

The following useful lemma is proven by Lun{\o}e-Nielsen and Rognes, as part of the proof of Proposition 2.2 in their paper \cite{MR3007679}:
\begin{lemma}\label{lnr lemma}
Suppose that $\{X_{i}\}$ is a sequence in the stable homotopy category. Suppose that each $X_i$ is a bounded below spectrum and each graded $\mathbb{F}_p$-module $H_*(X_i;\mathbb{F}_p)$ is finite type. Then
$ R^1\lim_i\pi_*((X_i)_p^{\wedge})=0.$
\end{lemma}

Given a prime number $p$, recall that an abelian group is said to be {\em $p$-reduced} if it has no nonzero infinitely $p$-divisible elements.
\begin{lemma}\label{finiteness lemma 1}
Let $p$ be a prime number, let $L$ be an integer, and let $\{X_{i}\}$ be a sequence of $p$-complete bounded-below spectra. Make the following assumptions:
\begin{enumerate}
\item The derived limit $R^1\lim_i{} \pi_s\left(X_i\right)$ is trivial for all integers $s$.
\item The mod $p$ reduction of the abelian group $\pi_s(\holim_i X_i)$ is finite for all $s\geq L$. x
\end{enumerate}
Then the abelian group $\lim_i k(n)_t(H\pi_sX_i)$ is finite for all integers $s\geq L$ and all integers $t$.
\end{lemma}
\begin{proof}
Let $\tilde{k}(n)$ be any spectrum equipped with a map $\tilde{k}(n)\rightarrow k(n)$ such that
 $\pi_*(\tilde{k}(n)) = \mathbb{Z}_{(p)}[v_n]$ as a graded abelian group,
and the map $\pi_*(\tilde{k}(n))\rightarrow \pi_*(k(n))$ is the reduction modulo $p$ map.
For example, $\tilde{k}(n)$ can be the spectrum obtained from $BP$ by using $BP$-module cells to cone off everything in the ideal in $BP_*$ generated by the elements $v_i$ with $i\neq n$.

By the first hypothesis and by the Milnor sequence for the homotopy groups of a sequential homotopy limit, we have isomorphisms
\begin{align*}
 k(n)_t(\holim_i H\pi_sX_i) 
  &\cong k(n)_t(H\lim_i \pi_sX_i) \\
  &\cong k(n)_t(H\pi_s\holim_i X_i) \\
  &\cong \tilde{k}(n)_t\left( S/p \smash H\pi_s\holim_i X_i \right).\end{align*}
The spectrum $S/p \smash H\pi_s\holim_i X_i$ is a two-stage Postnikov system, with
\begin{align*}
 \pi_j\left( S/p \smash H\pi_s\holim_i X_i\right)
  &\cong \left\{ \begin{array}{ll} 
   \mathbb{F}_p\otimes_{\mathbb{Z}} \pi_s\holim_i X_i &\mbox{\ if\ } j=0 \\
   \left(\pi_s\holim_i X_i\right)[p] &\mbox{\ if\ } j=1 \\
   0 &\mbox{\ if\ } j\neq 0,1 ,\end{array}\right.
\end{align*}
where the notation $G[p]$ denotes the $p$-torsion subgroup of an abelian group $G$.

Since the mod $p$ reduction of $\pi_s\holim_i X_i$ was assumed finite, we know that $\pi_0\left( S/p \smash H\pi_s\holim_i X_i\right)$ is finite. Since $\pi_s\holim_i X_i$ is $p$-reduced\footnote{The $p$-reducedness of $\pi_s\holim_i X_i\cong \lim_i\pi_sX_i$ follows from the fact that $p$-reducedness is a property preserved by sequential limits, and the fact that each $X_i$ is $p$-complete, hence has $p$-reduced homotopy groups. See Proposition 2.5 of \cite{MR551009} and VI.3.4(i) of \cite{MR0365573} for this last implication.}, the group $\pi_1\left( S/p \smash H\pi_s\holim_i X_i\right)$ is also finite. Hence $k(n)_t(\holim_i H\pi_sX_i)$ is $\tilde{k}(n)_t$ of a finite Postnikov system whose layers are each Eilenberg-Mac Lane spectra of finite abelian groups. Consequently $k(n)_t(\holim_i H\pi_sX_i)$
 is finite.

Theorem \ref{adams thm on holims} now implies that $\lim_i k(n)_t(H\pi_sX_i)$ is a quotient of the finite group $k(n)_t(\holim_i H\pi_sX_i)$. Hence $\lim_i k(n)_t(H\pi_sX_i)$ is finite.
\end{proof}

\begin{theorem}\label{thm: limit theorem}
Let $p$ be a prime number, let $M$ be an integer, and let $n$ be a positive integer. Suppose we have a sequence $\{X_{i}\}$ of bounded-below $p$-complete spectra  satisfying the following conditions:
\begin{itemize} 
\item $R^1\lim_i\pi_*X_i$ is trivial.
\item There exists an integer $L$ such that the mod $p$ reduction of the abelian group $\pi_s(\holim_i X_i)$ is finite for all $s\geq L$. 
\item For each integer $i$, the spectrum $X_i$ satisfies condition $H(M)$.
\item $k(n)_*(X_i)$ is simple $v_n$-torsion for all $i\in\mathbb{N}$.
\item For each $i$, the graded $\mathbb{F}_p$-module $H_*(X_i;\mathbb{F}_p)$ is finite type. 
\end{itemize}
Then the homotopy limit $\holim_i X_i$ is $K(n)$-acyclic.
\end{theorem}
\begin{proof} 
Same argument as in Theorem \ref{main thm prod case}, using spectral sequence \eqref{ss f43209} in place of spectral sequence \eqref{ss 1000000}.
There is one point worth commenting on: in the proof of Theorem \ref{main thm prod case}, we invoked the finiteness of the product $\prod_{i\in I} k(n)_t\left( \Sigma^s H\pi_s(X_i)\right)$ to bound the lengths of nonzero differentials hitting a given bidegree. In the present setting, we instead need to know that $\lim_i k(n)_t\left( \Sigma^s H\pi_s(X_i)\right)$ is finite. This is a consequence of Lemmas \ref{lnr lemma} and \ref{finiteness lemma 1}. We remark that this argument requires the assumption that each $X_i$ is $p$-complete (because of the $p$-completion appearing in Lemma \ref{lnr lemma}), which we did not need for Theorem \ref{main thm prod case}.
\end{proof}

\begin{remark}
Let $\{X_{i}\}$ be a sequence of spectra satisfying the hypotheses of Theorem \ref{thm: limit theorem}. Then since  $k(n)_*(X_i)$ is simple $v_n$-torsion, it is clear that $K(n)_{*}(X_{i})=0$ for each $i$, so $\lim_{i}K(n)_{*}X_{i}=0$. It is therefore a consequence of Theorem \ref{thm: limit theorem} that 
\[ \lim_{i}K(n)_{*}X_{i}=K(n)_{*}(\lim_{i}X_{i}),\]
i.e., under these hypotheses, $K(n)$-homology commutes with a {\em non-uniformly-bounded-below} homotopy limit.
\end{remark}
Fix a prime $p$ as well as integers $M$ and $n\ge 1$ for the remainder of this section. 
\begin{definition}\label{def of Km-amenability}
Let $\{X_{i}\}$ be a sequence in the stable homotopy category. We say that the sequence $\{X_{i}\}$ is {\em $K(n)$-amenable with parameter $M$} if the following conditions are all satisfied:
\begin{enumerate}
\item \label{amenable 1} each $X_i$ is bounded below and $p$-complete,
\item \label{amenable 2} $H_*(X_i;\mathbb{F}_p)$ is finite type for all $i$, 
\item \label{amenable 3} the sequence of graded $\mathbb{F}_p$-vector spaces $\{H(H_*(X_i ; \mathbb{F}_p),Q_{n})\}$ is pro-isomorphic to zero, 
\item \label{amenable 4} there exists an integer $L$ such that for each $s\ge L$ the groups $\pi_{s}(\lim_{i}X_{i})/p$ are finite, and
\item \label{amenable 5} for each $i$, the spectrum $X_i$ satisfies condition $H(M)$. 
\end{enumerate}
If the parameter $M$ is clear from context, then we simply say that the sequence $\{X_{i}\}$ is {\em $K(n)$-amenable}. 
\end{definition}

\begin{remark}\label{pro-mono}
Recall that a map $V\to W$ of pro-objects in graded $\mathbb{F}_p$-modules is a monomorphism if its kernel is pro-isomorphic to the zero sequence. A sequence $\{Z_{i}\}$ of abelian groups is pro-isomorphic to zero if for each $s$ there exists a $t\ge s$ such that $Z_{t}\to Z_{s}$ is the zero map. Hence a map of sequences $\{ V_i\}\to \{W_i\}$ is a pro-monomorphism if for each $s$ there exists a $t\ge s$ such that \[\ker (V_t \rightarrow W_t)\to \ker (V_s \rightarrow W_s)\] is the zero map. 
\end{remark}
\begin{lemma}\label{lem: pro-mono}
If $\{X_i\}$ is a $K(n)$-amenable sequence with parameter $M$, then the canonical map
\[ \{ k(n)_*(X_i)\}\to \{k(n)_*(X_i^{<M})\}\]
is a monomorphism in the category of pro-objects in graded $\mathbb{F}_p$-modules. 
\end{lemma}

\begin{proof}
The naturality of the edge homomorphism in the Adams spectral sequence for $k(n)\wedge X_{i}$ provides a morphism of pro-objects in the category of graded $\mathbb{F}_p$-vector spaces
\begin{align}\label{map 309599} \{ k(n)_*(X_i)\} &\to \{ \mathrm{hom}_{E(\tau_n)}(\mathbb{F}_p,H_*(X_i ;\mathbb{F}_p))\}.\end{align}
Each element of the kernel of the edge map 
$k(n)_*(X_i) \to \mathrm{Hom}_{E(\tau_n)}(\mathbb{F}_p,H_*(X_i ;\mathbb{F}_p))$,
can be represented by an element in the $E_2$-page of the Adams spectral sequence converging to $k(n)_*(X_i)$. That element is in $\Ext_{E(\tau_n)}^{s,*}(\mathbb{F}_p,H_*(X_i;\mathbb{F}_p))$ for some $s>0$. Hence, by Corollary \ref{input of Adams for k(n) smash X}, the pro-triviality of $\{ H(H_*(X_i;\mathbb{F}_p),Q_n)\}$ implies that the map \eqref{map 309599} of pro-vector-spaces is a monomorphism.

Since there is a natural canonical inclusion 
\[ \textup{inc} \colon \thinspace \Hom_{E(\tau_{n)}}(\mathbb{F}_{p},H_{*}(X_{i};\mathbb{F}_{p}))\subset H_*(X_i;\mathbb{F}_p)\]
we have a natural monomorphism of pro-objects in the category of graded $\mathbb{F}_p$-vector spaces
\[ 
    \{k(n)_*(X_i)\} \to \{ H_*(X_i;\mathbb{F}_p)\}.
\]
By naturality, the diagram 
\begin{equation}\label{diag m4509m}
	\xymatrix{
	 k(n)_{*}(X_{i})\ar[r] \ar[d] & \Hom_{E(\tau_{n})}(\mathbb{F}_{p},  H_*(X_i;\mathbb{F}_p))\ar[r]^(.6){\textup{inc}} \ar[d] & H_*(X_i;\mathbb{F}_p) \ar[d] \\ 
	  k(n)_{*}(X_{i}^{<M})\ar[r] & \Hom_{E(\tau_{n})}(\mathbb{F}_{p},  H_*(X_i^{<M};\mathbb{F}_p)) \ar[r]^(.6){\textup{inc}} & H_*(X_i^{<M};\mathbb{F}_p)  
	}
\end{equation}
commutes. We have shown that the composite of the top horizontal maps in diagram \eqref{diag m4509m} is a pro-monomorphism. The right-hand vertical map in \eqref{diag m4509m} is a levelwise monomorphism, since each spectrum $X_i$ satisfies condition $H(M)$. Hence the left-hand vertical map in \eqref{diag m4509m} must also be a pro-monomorphism, as claimed.
\end{proof}
Finally, we prove the main theorem. 
\begin{theorem}\label{thm on holim of margolis-acyclics with lim1 vanishing}
Suppose $\{X_{i}\}$ is a $K(n)$-amenable sequence with parameter $M$. 
Then $\holim_{i}X_{i}$ is $K(n)$-acyclic. 
\end{theorem} 
\begin{proof}
Essentially the same argument as that of Theorem \ref{main thm prod case}, although some care is required to establish the horizontal vanishing line in spectral sequence \eqref{ss f43209}. Our argument for that vanishing line is as follows.
Recall (e.g. from Scholie 3.5 in the appendix of \cite{MR0245577}) that a map of sequences of abelian groups $\{ A_i \} \rightarrow \{ B_i\}$ is pro-zero if and only if, for every integer $i$, there exists an integer $j(i)\geq i$ such that the composite map of abelian groups $A_{j(i)}\rightarrow A_i\rightarrow B_i$ is zero.

For each integer $k$, let $\gamma(k)$ be the minimum of the connectivities of the spectra $X_0, \dots ,X_k$, so that $Wh^{\gamma(k)}(X_i) \cong X_i$ for all $i\leq k$. By Lemma \ref{lem: pro-mono}, the map of sequences of graded abelian groups $\{ k(n)_*Wh^{M}(X_i)\} \rightarrow\{ k(n)_*(X_i)\}$ is pro-zero, i.e., zero in the pro-category of graded abelian groups\footnote{To be clear: ``zero in the pro-category of graded abelian groups'' is a stronger condition than ``each individual degree is pro-zero in the category of abelian groups.'' Lemma \ref{lem: pro-mono} yields the stronger of these two conditions.}. Consequently, for each integer $i$, there exists an integer $j(i)\geq i$ such that $k(n)_*\Wh^M(X_{j(i)}) \rightarrow k(n)_*(X_i)$ is zero. 

Hence, given an element $x = (x_0, x_1, \dots )$ of $\lim_i k(n)_t\Wh^M(X_i)$, for each given initial subsequence $(x_0, \dots ,x_k)$ of $x$, the image of $x_{j(k)}\in k(n)_t\Wh^M(X_{j(k)})$ in $k(n)_t\Wh^{\gamma(j(k))}(X_{k})$ is zero. By the compatibility of the elements of the sequence $(x_0,x_{1}, \dots)$, the image of $x_i$ in $k(n)_t\Wh^{\gamma(j(k))}(X_k)$ is zero for all $i\leq k$. Consequently the map 
\[ \lim_i k(n)_t\Wh^M(X_i) \rightarrow \underset{s}\colim \lim_i k(n)_t\Wh^s(X_i)\]
sends every initial subsequence of $(x_0, x_1, \dots)$ to zero.  
Hence $x$ maps to zero in the colimit $\colim_s \lim_k k(n)_t\Wh^M(X_k)$, by the same argument (using finite-dimensionality of $\lim_{i\in I}k(n)_t(H\pi_sX_i)$ for all $s\geq M$) as in the proof of Theorem \ref{thm on holim of margolis-acyclics with lim1 vanishing}.
\end{proof}

\section{A higher chromatic height analogue of Mitchell's theorem}\label{application}
In this section, we give a sample application of the main theorem of this paper, Theorem \ref{thm on holim of margolis-acyclics with lim1 vanishing}. Our sample application is to the Morava $K$-theory of the topological periodic cyclic homology
and the topological negative cyclic homology 
of $B\langle n\rangle$, where $B\langle n\rangle$ is any $p$-primary $E_{2}$ form of $BP\langle n\rangle$ satisfying Running Assumption \ref{assumption}. See Definition \ref{forms} for the definition of a ``$p$-primary $E_{2}$ form of $BP\langle n\rangle$.''

In \cite{Mit90}, Mitchell proved that  $K(m)_*(K(\mathbb{Z}))\cong 0$ for $m\ge 2$, and consequently $K(m)_*(K(R))\cong 0$ for any $H\mathbb{Z}$-algebra $R$. 
We might consider the following higher chromatic height analogue of Mitchell's theorem. 
\begin{question}\label{main q 4}
Suppose $n$ is some integer, $n\in [-1,\infty)$ and $B\langle n \rangle$ is a $p$-primary $E_{2}$ form of $BP\langle n\rangle$. If $R$ is a $B\langle n \rangle$-algebra spectrum, then does $K(m)_*(K(R))$ vanish for all $m\ge n+2$?
\end{question}
This gives an upper bound on the chromatic complexity of the algebraic $K$-theory of a $B\langle n\rangle$-algebra. By Ravenel \cite[Thm. 2.1(d),(f),(i)]{MR737778}, we know that 
\[ K(m)_*(B\langle n \rangle)\cong 0 \]
for $m\ge n+1$, and every $B\langle n\rangle$-algebra is likewise $K(m)$-acyclic. 
Consequently a positive answer to Question \ref{main q 4} implies that, if there is a ``red-shift'' in algebraic K-theory of a $B\langle n\rangle$-algebra spectrum, then this shift is a shift of at most one. 

The main goal of this section is to answer Question \ref{main q 4} for all $(n,p)$ such that there exists a $p$-primary $E_2$ form $B\langle n\rangle$ of $BP\langle n\rangle$ satisfying Running Assumption \ref{assumption}. Question \ref{main q 4} is already known to have a positive result for $n=-1$ by Quillen \cite{MR0315016}, for $n=0$ by Mitchell \cite{Mit90}, and for $n=1$ and $p\ge 5$ by Ausoni--Rognes \cite{AR02} after $p$-completion. Consequently our main new contributions are the cases $n=1,2$ at $p=3$, though we do not know of obstacles to working out the same calculations at $p=2$. 
\begin{remark}\label{conditional remark}
Our results rely on the calculation, by Bruner and Rognes in \cite{BR05}, of the continuous homology $H^c_*(TC^-(BP\langle n\rangle); \mathbb{F}_p)$, for all primes $p$ and integers $n$. The proof uses the fact that when $BP\langle n\rangle$ is an $E_{\infty}$ ring spectrum than it admits certain Kudo--Araki--Dyer--Lashof operations. 

If $n$ is sufficiently large, then the spectrum $BP\langle n\rangle$ is known to {\em not} admit an $E_{\infty}$ ring structure: at $p=2$, this is a theorem of Lawson \cite{MR3862946}, and at $p>2$, a theorem of Senger \cite{Sen24}.
This seems to suggest that the main theorem in this section, Theorem \ref{main application}, cannot have consequences for large integers $n$. On the other hand, the spectrum $BP$ was shown to admit an $E_4$ ring spectrum model by Basterra--Mandell \cite{MR3065177}. Since the first draft of this paper appeared in pre-print form, Hahn--Wilson \cite{HW22} have proven that there are specific forms of $BP\langle n\rangle$ built from $MU_{(p)}$, which are $E_3$ ring spectra, and consequently $THH$ of such a form of $BP\langle n\rangle$ is an $E_2$-ring spectrum. If the calculations of Bruner--Rognes \cite{BR05} can be made to work using only the $E_m$ Dyer--Lashof--Kudo--Araki operations, for appropriate values of $m$, 
rather than the classical ($E_{\infty}$) Dyer--Lashof--Kudo--Araki operations, then our Theorem \ref{main application} would yield a positive answer to Question \ref{main q 4} for all primes $p$ and heights $n$.
\end{remark}
\begin{remark}
Since the first draft of this paper was posted, the papers \cite{CMNN24,LMMT24} appeared. These papers establish that Question \ref{main q 4} has a positive answer at all primes for any $p$-primary $E_1$ form of $BP\langle n\rangle$. Their work uses entirely different methods from ours, and our work applies to approximations of algebraic K-theory such as $TC^{-}$ and $TP$, whereas their approach works directly with algebraic $K$-theory. 
\end{remark}

\subsection{Trace methods}
Now we briefly recall the setup for topological periodic cyclic homology and topological negative cyclic homology. 
Let $R$ be an $E_1$ ring spectrum. Write $\mathbb{T}$ for the circle, regarded as the compact Lie group of unit vectors in the complex numbers. 
It is well-known that the topological Hochschild homology of $R$, denoted $THH(R)$, has a canonical 
action of the circle group  $\mathbb{T}$. 
Let $E\mathbb{T}=S(\mathbb{C}^{\infty})$ be the unit sphere in $\mathbb{C}^{\infty}$ where $\mathbb{T}$ acts on $\mathbb{C}^{\infty}$ coordinate-wise. Therefore, $E\mathbb{T}$ is a $\mathbb{T}$-space with free $\mathbb{T}$-action whose underlying space is contractible. We also consider the homotopy cofiber sequence 
\[ E\mathbb{T}_+ \to S^0 \to \widetilde{E}\mathbb{T}\]
in topological spaces where the map $E\mathbb{T}_+ \to S^0$ is induced by the unique map $E\mathbb{T}\to *$ of unpointed topological spaces. 

We define \emph{topological negative cyclic homology} as the homotopy fixed-point spectrum
\[ TC^{-}(R):=THH(R)^{h\mathbb{T}} = F(E\mathbb{T}_+,THH(R))^{\mathbb{T}}\]
and \emph{topological periodic cyclic homology} as 
the Tate spectrum
\[ TP(R):=THH(R)^{t\mathbb{T}} =\left ( \widetilde{E}\mathbb{T} \wedge F(E\mathbb{T}_+,THH(R)) \right )^{\mathbb{T}}.\]

There is a homological $\mathbb{T}$-homotopy fixed point spectral sequence 
\begin{align}\label{HFPSS} E_2^{*,*}(R)=H^*(\mathbb{T},H_*(THH(R);\mathbb{F}_p))\Rightarrow H^c_*(TC^{-}(R);\mathbb{F}_p).\end{align}
The notation $H^c_*$ refers to {\em continuous homology,} and is defined so that
\[ H^c_*(TC^{-}(R);\mathbb{F}_p)= \lim_k H_*(TC^{-}(R)[k];\mathbb{F}_p),\] 
where 
\[ TC^{-}(R)[k]= F\left((E\mathbb{T}^{(2k)})_+,THH(R)\right)^{\mathbb{T}},\]
and where $E\mathbb{T}^{(2k)}$ is the $2k$-skeleton of the standard presentation for $E\mathbb{T}$ as an $\mathbb{T}$-CW complex. Explicitly, we let $E\mathbb{T}^{(2k)}=S(\mathbb{C}^{k+1})$, and it is obtained from $E\mathbb{T}^{(2k-2)}$ by attaching a single $\mathbb{T}$-cell of dimension $2k$. We refer the reader to \cite[Sec. 2]{BR05} for a more detailed account. 

There is also a homological $\mathbb{T}$-Tate spectral sequence 
\begin{align}\label{TSS}\widehat{E}_2^{*,*}(R)=\hat{H}^*(\mathbb{T},H_*(THH(R);\mathbb{F}_p))\Rightarrow H^c_*(TP(R);\mathbb{F}_p),\end{align}
with abutment 
$H^c_*(TP(R);\mathbb{F}_p):= \underset{k}{\lim\thinspace}H_*(TP(R)[k];\mathbb{F}_p)$.
The filtration 
\[ \dots \rightarrow TP(R)[1] \rightarrow TP(R)[0] \rightarrow TP(R)[-1] \rightarrow \dots\]
is the {\em Greenlees filtration} \cite{MR908451}, and is given by 
\begin{align}\label{Greenlees filtration} TP(R)[k]:= \left( \widetilde{E}\mathbb{T}/ \widetilde{E}\mathbb{T}_{-2k}\wedge F(E\mathbb{T}_+,THH(R))\right)^{\mathbb{T}}.\end{align}
The definition of the $\mathbb{T}$-equivariant spectrum $\widetilde{E}\mathbb{T}_{2k}$ depends on whether $k$ is negative or not, as follows.
\begin{description}
\item[If $k\geq 0$] then $\widetilde{E}\mathbb{T}_{2k}$ is the cofiber of the map $E\mathbb{T}^{(2k)}_+\to S^0$, again induced by the unique unpointed $\mathbb{T}$-equivariant map  $E\mathbb{T}^{(2k)}\to *$.
\item[If $k< 0$] then $\widetilde{E}\mathbb{T}_{2k}$ is the Spanier-Whitehead dual of $\widetilde{E}\mathbb{T}_{-2k+2}$.
\end{description}

We will need to make use of two more related spectral sequences. By considering the filtrations on $E\mathbb{T}$ and on $\widetilde{E}\mathbb{T}$ only in a range of dimensions, there is also a spectral sequence
\begin{align}\label{tHFPSS} E_2^{*,*}(R)[k]=P_{k+1}(t)\otimes H_*(THH(R);\mathbb{F}_p)\Rightarrow H_*(TC^{-}(R)[k];\mathbb{F}_p)\end{align}
called the \emph{approximate homotopy fixed point spectral sequence}. There is also a spectral sequence 
\begin{align}\label{tTSS} \widehat{E}_2^{*,*}(R)[k]=P(t^{-1})\{t^{k}\}\otimes H_*(THH(R);\mathbb{F}_p)\Rightarrow H_*(TP(R)[k];\mathbb{F}_p),\end{align}
called the \emph{approximate Tate spectral sequence}, where 
\[P(t^{-1})\{t^{k}\}=\{xt^k: x\in P(t^{-1}) \}.\] 

Each of these four spectral sequences strongly converge when $\pi_*(THH(R))$ is bounded below and $H_m(THH(R);\mathbb{F}_p)$ is finite for all $m$. All of these spectral sequences are spectral sequences of $A_*$-comodules; this follows from the same proof as given in \cite[Prop. 2.1]{BR05} in the case of the spectral sequence \eqref{HFPSS}. When $H_m(THH(R);\mathbb{F}_p)$ is finite for all $m$ and bounded below, we also know that these are all spectral sequences of $E(Q_m)$-modules, where $Q_m$ is the Milnor primitive \cite{MR0099653} in the Steenrod algebra $A$, dual to the indecomposable $\btau_m\in A_*$.  In other words, the differentials are $Q_m$-linear. 

More generally, for a generalized homology theory $E_*$ we write 
\begin{align*}
E_*^c(TP(R)):= & \lim_{k} E_*(TP(R)[k])  \text{ and } \\
E_*^c(TC^{-}(R)):= & \lim_{k} E_*(TC^{-}(R)[k]).
\end{align*}
The mod $p$ homology of topological Hochschild homology comes equipped with an operator $\sigma$, sometimes called the {\em circle operator}. The circle operator $\sigma$ is induced by the canonical map 
\[\sigma \colon \thinspace  R\wedge \mathbb{T} \to R\wedge \mathbb{T}_+  \to THH(R)\] 
so that $\sigma  x\in H_*(THH(R);\mathbb{F}_p)$ is the image of 
\[ x\otimes \iota \in  H_*(R;\mathbb{F}_p)\otimes H_*(\mathbb{T};\mathbb{F}_p) \cong H_*(R\wedge \mathbb{T};\mathbb{F}_p) \subset H_*(R\wedge \mathbb{T}_+;\mathbb{F}_p)\] 
where $\iota$ is the canonical generator of $H_1(\mathbb{T};\mathbb{F}_p)$. Note that this operator is compatible with the canonical $\mathbb{T}$-action
\[ \alpha \colon  \thinspace  THH(R)\wedge \mathbb{T}\to  THH(R) \wedge \mathbb{T}_+\to THH(R) \]
in the sense that 
\[ \sigma(x) =\alpha (\eta(x)\otimes \iota)\]
where $\eta \colon \thinspace H_*(R) \to H_*(THH(R))$ is the unit map of $R$ algebra $THH(R)$ and we write $\alpha$ for the induced map 
\[ \alpha \colon \thinspace H_*(THH(R))\otimes H_*(\mathbb{T})\to  H_*( THH(R)\wedge \mathbb{T}_+)\to  H_*(THH(R)) \]
by abuse of notation. Also, by \cite[Prop. 5.10]{AR05} we know that
\[\sigma ( x \cdot y) = x \cdot \sigma (y) + (-1)^{|y|} \sigma (x) \cdot y \]
for $x,y\in H_*(R;\mathbb{F}_p)$, so in particular, $\sigma (x^p)=0$ whenever $|x|$ is even. In particular, $\sigma(1)=0$.
 
 \subsection{A higher height Mitchell theorem}\label{A higher height Mitchell theorem}
We now recall that by Angeltveit--Rognes \cite[Thm. 5.12]{AR05}, for odd primes $p$, there is an isomorphism 
\[ H_*(THH(BP\langle n \rangle);\mathbb{F}_p)\cong 
				H_*(BP\langle n \rangle;\mathbb{F}_p)\otimes_{\mathbb{F}_p} E(\sigma \bxi_1,\sigma \bxi_2,\ldots ,\sigma \bxi_{n+1})\otimes_{\mathbb{F}_p} P(\sigma \btau_{n+1})
\]
of $H_*(BP\langle n \rangle;\mathbb{F}_p)$-algebras. The left $A_*$-coaction 
\[ \nu_n\colon\thinspace  H_*(THH(BP\langle n \rangle);\mathbb{F}_p)\to A_*\otimes_{\mathbb{F}_p} H_*(THH(BP\langle n \rangle);\mathbb{F}_p)\]
is given as follows:
\begin{itemize}
\item On elements in $H_*(BP\langle n \rangle;\mathbb{F}_p) \subseteq H_*(THH(BP\langle n \rangle);\mathbb{F}_p)$, $\nu_n$ is simply the restriction of the coproduct of $A_*$ to $H_*(BP\langle n \rangle;\mathbb{F}_p)\subset A_*$.
\item On elements in $H_*(THH(BP\langle n \rangle);\mathbb{F}_p)$ of
 the form $\sigma \bxi_i$ for $1\le i\le n$, and on the element $\sigma \btau_{n+1}$, $\nu_n$ is given by the formula
\begin{align}\label{nu sigma}
 \nu_n(\sigma x)=(1\otimes \sigma)(\nu_n(x))
 \end{align}
from \cite[Eq. 511]{AR05}.
\item On the remaining elements of $H_*(THH(BP\langle n \rangle);\mathbb{F}_p)$, the coaction $\nu_n$ is given by the formula $\nu_n(xy)=\nu_n(x)\nu_n(y)$.\end{itemize}

In particular, because $\sigma(\bxi_i^{p^j})=0$ for $j\ge 1$, formulas \eqref{xi coprod} and \eqref{nu sigma} imply that the elements $\sigma \bxi_k$ are comodule primitives. Formulas \eqref{tau coprod} and \eqref{nu sigma} imply that 
\begin{align}
\label{sigma btau coaction} \nu(\sigma\btau_{n+1}) &=1\otimes \sigma \btau_{n+1} + \btau_0\otimes \sigma \bxi_{n+1},\end{align}
which also appears in Theorem 5.12 of \cite{AR05}.
 
Next, we recall the computation of continuous homology of topological negative cyclic homology and topological periodic cyclic homology of $BP\langle n\rangle$.
Bruner and Rognes \cite[Proposition 6.1]{BR05} prove that the $E_{\infty}$ term of the homological homotopy fixed-point spectral sequence \eqref{HFPSS} admits an isomorphism
\[ 
		E_\infty^{*,*}(BP\langle n \rangle)\cong  \left( P(t)\otimes_{\mathbb{F}_p} M_1\otimes_{\mathbb{F}_p} M_2 \right) \oplus T
\]
where 
\begin{align} 
\label{M1} M_1= &E(\tau_{j+1}^{\prime}\mid j\ge n+1 ),\\
\label{M2} M_2=&P(\bxi_{j+1}|j\ge n+1)\otimes P(\bxi_j^p \mid1\le j \le n+1)\otimes E(\bxi_j^{p-1}\sigma \bxi_j \mid 1\le j \le n+1),
\end{align}
and where $T$ consists of classes $x$ in filtration $s=0$ with $tx=0$ and with
\[ \tau_{k+1}^{\prime}=\btau_{k+1}-\btau_k(\sigma \btau_k)^{p-1}\]
for $k\ge m+1$. 

One can then easily deduce the computation of the $E_{\infty}$ term 
\[ 
		\widehat{E}_\infty^{*,*}(BP\langle n \rangle)\cong P(t,t^{-1})\otimes_{\mathbb{F}_p} M_1\otimes_{\mathbb{F}_p} M_2 
\]
of the homological $\mathbb{T}$-Tate spectral sequence \eqref{TSS} for $BP\langle n\rangle$, where $M_1$ and $M_2$ are as defined in \eqref{M1} and \eqref{M2}. There are no possible additive extensions because the abutment is a graded $\mathbb{F}_p$-vector space. 
It will also be useful to record the structure of the $E_{\infty}$-page of the approximate Tate spectral sequence \eqref{tTSS}. There is an isomorphism of $E(Q_m)$-modules, for all $m\ge n+2$,
\begin{align}\label{truncated Einf}
	\widehat{E}_\infty^{*,*}(BP\langle n \rangle)[k]=& \left ( P(t^{-1})\{t^{k-1}\} \otimes_{\mathbb{F}_p} M_1 \otimes_{\mathbb{F}_p} M_2 \right )\oplus V_k(n)\{t^k\}
\end{align}
where 
\begin{align} 
 \label{vnk} V_k(n)= &H_*(THH(BP\langle n\rangle);\mathbb{F}_p)/\im (d_2^{2-2k,*}). 
 \end{align}
Here $d^{2-2k,*}_2$ denotes the differential in the approximate Tate spectral sequence \eqref{tTSS}, regarded as a map 
\[ d_2^{2-2k,*} \colon \thinspace H_*(THH(BP\langle n\rangle;\mathbb{F}_p) \to  H_{*+1}(THH(BP\langle n\rangle;\mathbb{F}_p).\]
Note that the spectral sequence \eqref{tTSS} collapses at the $E_3$-term as a consequence of \cite[Prop. 6.1]{BR05}, so we do not need to consider the image of longer differentials in our description of $V_k(n)$. 

Now we prove a lemma that will be useful for the main Margolis homology calculation. 
\begin{lemma}\label{little Margolis lemma}
Suppose $k$ is an integer, and suppose that
\[ 
  \xymatrix{
    \dots \ar@{^(->}[r]  &  F_2 \ar@{^(->}[r] \ar[d] &  F_1 \ar@{^(->}[r] \ar[d] &   F_0 \ar[d]  \\
    \dots \ar@{^(->}[r] &  F_2^{\prime} \ar@{^(->}[r] &  F_1^{\prime} \ar@{^(->}[r] &   F_0^{\prime} \\
  }
\]
is a map of filtered graded $E(Q_m)$-modules satisfying the following conditions: \begin{enumerate}
\item $F_j=0$ for $j\ge k+2$,
\item $F_j^{\prime}=0$ for $j\ge k+1$, 
\item $F_j=F_j^{\prime}$ for $j<k$,
\item and $H(F_j/F_{j+1},Q_m)$ and $H(F^{\prime}_{j}/F^{\prime}_{j+1},Q_m)$ are trivial for all $j\le k$.
\end{enumerate}
Then the map of graded $\mathbb{F}_p$-vector spaces induced in Margolis homology
\[H(F_{0},Q_m)\longrightarrow H(F_{0}^{\prime},Q_m)\]
is the zero map.
\end{lemma}
\begin{proof}
By the long exact sequences in Margolis homology induced by the short exact sequences
\[ \xymatrix{0 \ar[r]  & F_{j+1} \ar[r] & F_{j}\ar[r] & F_{j}/F_{j+1} \ar[r]  & 0
  }
\]
\[
    \xymatrix{ 0 \ar[r] & F^{\prime}_{j+1} \ar[r] & F^{\prime}_{j}\ar[r] & F^{\prime}_{j}/F^{\prime}_{j+1} \ar[r] & 0
  }
\]
and the assumptions that $H(F_j/F_{j+1},Q_m)=H(F^{\prime}_j/F^{\prime}_{j+1},Q_m)=0$ for $j\le k$, we know that the induced maps 
$H(F_{j+1},Q_m)\overset{\cong}{\longrightarrow} H(F_j,Q_m)$ and 
$H(F^{\prime}_{j+1},Q_m)\overset{\cong}{\longrightarrow} H(F^{\prime}_j,Q_m)$ 
are isomorphisms for $j\le k$. When $j=k$, the map of long exact sequences in Margolis homology associated to the map of short exact sequences 
\[
  \xymatrix{
    0 \ar[r] \ar[d] & F_{k+1} \ar[r] \ar[d] & F_{k}\ar[r] \ar[d] & F_{k}/F_{k+1} \ar[r] \ar[d] & 0 \ar[d] \\
    0 \ar[r] &0  \ar[r] & F^{\prime}_{k}\ar[r] & F^{\prime}_{k} \ar[r] & 0
  }
\]
has a subdiagram of the form 
\[
  \xymatrix{
   H(F_{k+1},Q_m)\ar[r]^{\cong} \ar[d] &  H(F_{k},Q_m)\ar[r] \ar[d] & 0\ar[d]  \\
     0 \ar[r] & H(F_k^{\prime},Q_m) \ar[r]^{\cong} & H(F_k^{\prime},Q_m)
  }
\]
so the map $H(F_k,Q_m)\to H(F^{\prime}_k,Q_m)$ is the zero map. Now the compatible isomorphisms \[ H(F_k,Q_m)\cong H(F_{k-1},Q_m)\cong H(F_{k-2},Q_m)\cong \dots \mbox{\ \ \ and} \]\[ H(F^{\prime}_k,Q_m)\cong H(F^{\prime}_{k-1},Q_m)\cong H(F^{\prime}_{k-2},Q_m)\cong \dots \] yield that $H(F_{j},Q_m)\longrightarrow H(F_{j}^{\prime},Q_m)$
is also the zero map for all $j<k$.
\end{proof}
We now apply this lemma in the main example of interest.
\begin{prop}\label{Margoliscomputaton}
Let $p$ be odd, and let $B\langle n\rangle$ be a $p$-primary $E_{2}$ form of $BP\langle n\rangle$. There is an isomorphism of pro-objects in abelian groups
\[ \{ H( H_* (TP(B\langle n\rangle)[k];\mathbb{F}_p),Q_m)\}_{k\in \mathbb{Z}} \cong 0 \]
for all $m\ge n+2$. Consequently, there is an isomorphism
\[ \underset{k}{\lim\thinspace}H( H_* (TP(B\langle n\rangle)[k];\mathbb{F}_p),Q_m)\cong 0 .\]
\end{prop}
\begin{proof}
We will show that each of the induced maps 
\[ H(H_* (TP(B\langle n\rangle)[k];\mathbb{F}_p),Q_m )\to  H(H_* (TP(B\langle n\rangle)[k-1];\mathbb{F}_p),Q_m)\]
is the zero map for $m\ge n+2$, which implies the result. 

First, there is a map of $E(Q_m)$-module spectral sequences
\[ 
  \xymatrix{
    P(t^{\pm 1}) \ar@{=>}[d] \ar[r] & P(t^{-1})\{t^k\} \otimes_{\mathbb{F}_p} H_*(THH(B\langle n \rangle ;\mathbb{F}_p)\ar@{=>}[d] \\
    H_*^c(TP(S);\mathbb{F}_p) \ar[r] & H_*(TP(B\langle n \rangle ;\mathbb{F}_p)[k])
  }
\]
where the first spectral sequence collapses for bidegree reasons. The $Q_m$-action is trivial on $t^j$ in $H_*^c(TP(S);\mathbb{F}_p)$, because $H_*^c(TP(S);\mathbb{F}_p)$ is concentrated in even degrees. Consequently $Q_m(t^j)=0$ in the abutment $H_*(TP(B\langle n \rangle ;\mathbb{F}_p)[k])$ of the approximate Tate spectral sequence \eqref{tTSS} for all $m$ and all $k \ge j>-\infty$. 

Our next task is to compute the Margolis homology $H(\mathbb{F}_p\{ t^j\} \otimes_{\mathbb{F}_p} M_1\otimes_{\mathbb{F}_p} M_2,Q_m)$ where $M_1$ and $M_2$ are the $E(Q_m)$-modules defined in \eqref{M1} and \eqref{M2}, and where $\mathbb{F}_p\{ t^j\}$ is the graded $\mathbb{F}_p$-vector subspace of $P(t^{-1})\{ t^k\}$ spanned by $t^j$. The $\mathbb{F}_p$-vector space $\mathbb{F}_p\{ t^j\}$ is equipped with the $E(Q_m)$-module structure in which $Q_m$ acts trivially. The tensor product $\mathbb{F}_p\{ t^j\} \otimes_{\mathbb{F}_p} M_1\otimes_{\mathbb{F}_p} M_2$ is then a $E(Q_m)$-module using the usual action of $Q_m$ on a tensor product of $E(Q_m)$-modules.

Recall that $M_1 \cong E(\tau^{\prime}_{n+1}, \tau^{\prime}_{n+2}, \dots)$. Write $\overline{M}_1$ for the exterior algebra on the generators $\tau^{\prime}_{n+1}, \tau^{\prime}_{n+2}, \dots$ {\em except} for $\tau^{\prime}_m$. We claim that $M_1 \cong E(\tau^{\prime}_m)\otimes \overline{M}_1$ as $E(Q_m)$-modules. This tensor splitting follows from the formula
\begin{align}\label{Q_m acting on tau} Q_m(\tau_k^{\prime})&=\bar{\xi}_{k-m}^{p^{k-m}},\end{align}
which in turn follows from the $A_*$-coaction on $\tau_k^{\prime}$ being given by
\begin{align}\label{nu on tau prime} \nu(\tau_k^{\prime}) &=1\otimes \tau_k^{\prime}+\sum_{i+j=k}\bar{\tau}_i\otimes \bar{\xi}_j^{p^i},\end{align}
as a consequence of \eqref{sigma btau coaction}.
 
As another consequence of \eqref{Q_m acting on tau}, $E(Q_m)$ acts freely on $E(\tau_m^{\prime})$. Hence $\mathbb{F}_p\{ t^j\} \otimes_{\mathbb{F}_p} M_1\otimes_{\mathbb{F}_p} M_2$ has a free $E(Q_m)$-module as a tensor factor. This is enough to conclude that $\mathbb{F}_p\{ t^j\} \otimes_{\mathbb{F}_p} M_1\otimes_{\mathbb{F}_p} M_2$ is free over $Q_m$, hence its Margolis $Q_m$-homology is trivial.

We then write $F_{\bullet}^{(k)}$ for the abutment of the approximate Tate spectral sequence \eqref{tTSS}, regarded as a filtered graded $E(Q_m)$-module. Consider the map of filtered graded $E(Q_m)$-modules
$F_{\bullet}^{(k+1)}\to F_{\bullet}^{(k)}$. By the calculation of Bruner and Rognes (see \eqref{HFPSS second mention}, above), for any integers $j,k$ satisfying $j<k$, there are isomorphisms
\begin{align} \label{HFPSS second mention} F_{j}^{(k)}/F_{j+1}^{(k)}&\cong \mathbb{F}_p\{t^j\}\otimes_{\mathbb{F}_p} M_1\otimes_{\mathbb{F}_p} M_2 . \end{align}
Hence there are identifications $F_j^{(k+1)}\cong F_j^{(k)}$ for $j<k$, and $F_j^{(k)}=0$ for $j>k$, and in the preceding paragraphs of this proof, we showed that the Margolis $Q_m$-homology of the quotients $F_j^{(k)}/F_{j+1}^{(k)}$ is trivial for $j\leq k$. That confirms all the hypotheses of Lemma \ref{little Margolis lemma}. The lemma then implies the claimed result.  
\end{proof}
\begin{lemma}\label{comodule primitive lemma}
Let $p$ be odd.
Fix a positive integer $k$. Suppose $L$ is a graded $A_*$-comodule such that $L\cong P(t^{-1})\{t^k\}\otimes_{\mathbb{F}_p} B\otimes_{\mathbb{F}_p} E,$
where 
\begin{itemize}
\item $B$ is isomorphic to a graded comodule subalgebra of the dual Steenrod algebra $A_*$, 
\item $E$ is a graded $A_*$-comodule, finite-dimensional as a $\mathbb{F}_p$-vector space, and consisting only of $A_*$-comodule primitives, 
\item 
and we have an isomorphism $P(t^{-1})\{t^k\} \cong H_*(TP(S)[k];\mathbb{F}_p)$ of $A_{*}$-comodules. 
\end{itemize}
Then the $A_*$-comodule primitives in $L$ are bounded above by $\sum_{i\in I} \left| e_i\right|$, where $\{ e_i: i\in I\}$ is a homogeneous $\mathbb{F}_p$-linear basis for $E$, and where $\left| e_i\right|$ is the degree of $e_i$.
\end{lemma}
\begin{proof} 
Recall that the $j$-th stage in the Greenlees filtration of $TP(S)$ has homology 
\[ H_*(TP(S)[j];\mathbb{F}_p)\cong P(t^{-1})\{t^j\}=\{xt^j:x\in P(t^{-1})\}\]
as a quotient of $H_*(TP(S);\mathbb{F}_p)\cong P(t^{\pm 1})$.
Hence we may consider $L$ as a filtered graded $\mathbb{F}_p$-vector space 
\[ L=L_{k}\supset L_{k-1}\supset \dots \supset L_{-\infty}=0,\]
by defining $L_j$ as
\[L_{j}=H_*(TP(S)[j];\mathbb{F}_p)\otimes_{\mathbb{F}_p} B\otimes_{\mathbb{F}_p} E.\]
We refer to this filtration as the $t$-weight filtration on $L$. Specifically, $t^j$ has $t$-weight $j$. (Note that higher $t$-weight corresponds to lower degree, since $t^j$ is in degree $-2j$ in homology.)

We may regard the $\mathbb{F}_p$-vector space quotient $L_{j}/L_{j-1}$ of $L_j$ also as a graded $\mathbb{F}_p$-submodule of $L$, by identifying the quotient $L_j/L_{j-1}$ with
\[ B\otimes_{\mathbb{F}_p} E\{t^j\}=\{zt^j:z\in B\otimes_{\mathbb{F}_p} E\}\subset L.\] 
Since $B$ is isomorphic to a comodule subalgebra of the dual Steenrod algebra, the only $A_*$-comodule primitive of $B$ is the element $1$. Writing 
\[ \psi_B\colon \thinspace B\to A_*\otimes_{\mathbb{F}_p} B\]
for the $A_{*}$-coaction, we have 
\[ \psi_B(x_i)=1\otimes x_i+ \sum_{j\ge 0}y_j^i\otimes z_j^i.\]
If $x_i\in B$ is not a $\mathbb{F}_p$-linear multiple of $1$, then the sum $\sum_{j\ge 0}y_j^i\otimes z_j^i$ must be nonzero, or otherwise $x_i$ would be an $A_{*}$-comodule primitive. We also know that $t\in P(t^{-1})\{t^k\}$ has coaction\footnote{The coaction \eqref{coaction on cp infty} is extremely classical, but curiously, we know very few places where it appears explicitly in the literature. One place is page 14 of \cite{MR4557878}. It is, however, derivable from Milnor's formula $\lambda^*(t) = \sum_{j\geq 0} t^{p^j} \otimes \xi_j$ for the adjoint $A_*$-coaction on the {\em cohomology} of $\mathbb{C}P^{\infty}$, from Lemma 6 in section 5 of \cite{MR0099653}. The idea is to identify positive-degree Tate cohomology with negative-degree homology, in such a way that the $A_*$-coaction on negative-degree homology of $\mathbb{C}P^{\infty}_{-\infty}$ agrees with the adjoint $A_*$-coaction on cohomology of $\mathbb{C}P^{\infty}$.} 
\begin{align}\label{coaction on cp infty} \psi_k(t) &=\sum_{j\ge 0}\overline{\xi}_j\otimes t^{p^j}\end{align}
in which terms involving a power $t^{p^j}$ with $p^j>k$ are omitted.
More generally, we can compute the coaction on $P(t^{-1})\{t^k\}$ by the formula
\[\psi_k(t^i)=(\sum_{j\ge 0}\xi_j\otimes t^{p^j})^i \]
in which the sum is taken in the graded $\mathbb{F}_p$-algebra $\mathbb{F}_p[t^{\pm 1}]$, and then terms with $t$-weight $>k$ are omitted.

Since $E$ consists only of $A_*$-comodule primitives, the natural map
\[\Hom_{A_*}(\mathbb{F}_p,M)\otimes_{\mathbb{F}_p} E \rightarrow \Hom_{A_*}(\mathbb{F}_p,M\otimes_{\mathbb{F}_p} E) \]
is an isomorphism for any $A_*$-comodule $M$.
Hence there is an isomorphism 
\begin{align}\label{iso 0965665} \Hom_{A_*}(\mathbb{F}_p,P(t^{-1})\{t^k\}\otimes_{\mathbb{F}_p} B\otimes_{\mathbb{F}_p} E) &\cong \Hom_{A_*}(\mathbb{F}_p,P(t^{-1})\{t^k\}\otimes_{\mathbb{F}_p} B)\otimes_{\mathbb{F}_p} E.\end{align}

The highest-degree element in the graded comodule $E$ is the product $e_{1}\cdot \ldots \cdot e_{n}$ of all the generators of the exterior algebra $E$. This product has degree $\sum_{i=1}^{n}|e_{i}|$. If we can show that $P(t^{-1})\{t^k\}\otimes_{\mathbb{F}_p} B$ has no nonzero $A_*$-comodule primitives in positive degrees, then isomorphism \eqref{iso 0965665} will tell us that $L$ has no $A_*$-comodule primitives in degree greater than $\sum_{i\in I} \left| e_i\right|$, finishing the proof of the theorem.

Even better: when $k\geq 2p-2$, neither $P(t^{-1})\{t^k\}$ nor $B$ have any comodule primitives in positive degrees. To see that this is the case for $P(t^{-1})\{t^k\}$, it is a matter of simple calculation, using the formula \eqref{coaction on cp infty}. To see that $B$ has no comodule primitives in positive degrees, one need only remember that $B$ was assumed to be a comodule subalgebra of the Steenrod algebra.

Continue to assume $k\geq 2p-2$. While neither $P(t^{-1})\{t^k\}$ nor $B$ have any comodule primitives in positive degrees, we need to verify that their tensor product $A_*$-comodule $P(t^{-1})\{t^k\} \otimes_{\mathbb{F}_p} B$ also has no comodule primitives in positive degrees. Suppose we have a homogeneous, positive-degree sum of the form $\sum_{i}t^i\otimes x_i\in P(t^{-1})\{t^k\}\otimes_{\mathbb{F}_p} B$, with $x_i\in B$, all but finitely many $x_i$ are zero, and the sum satisfies $\sum_{i}t^i\otimes x_i\not\in P(t^{-1})\{t^k\}$ and $\sum_{i}t^i\otimes x_i\not\in B$. We need to verify that $\sum_{i}t^i\otimes x_i$ is not an $A_*$-comodule primitive.

Write $\psi_B$ for the $A_*$-coaction map on $B$.
Write $1\otimes x_i + \sum_{j\geq 0} y_j^i \otimes z_j^i$ for $\psi_B(x_i)$, and write 
\[ \psi_{k,B}\colon\thinspace P(t^{-1})\{t^k\}\otimes_{\mathbb{F}_p} B\to A_*\otimes_{\mathbb{F}_p} P(t^{-1})\{t^k\}\otimes_{\mathbb{F}_p} B\]
for the $A_{*}$-coaction on $P(t^{-1})\{t^k\}\otimes_{\mathbb{F}_p} B$. 
The coaction map $\psi_{k,B}$ satisfies
\begin{align*} 
 \psi_{\ell,B}(x_{i}t^{i})
  &=   \psi_B(x_{i})\cdot \psi_{\ell}(t^{i}) \\
  &= (1\otimes x_{i}+ \sum_{j\ge 0}y_j^i\otimes z_j^i)(\sum_{j\ge 0}\overline{\xi}_j\otimes t^{p^j})^{i} 
  \end{align*}
in which the sum is taken in the graded $\mathbb{F}_p$-algebra $\mathbb{F}_p[t^{\pm 1}]\otimes_{\mathbb{F}_p} B$, and then terms with $t$-weight $>k$ are omitted.
The $t$-weight filtration on $L$ is a filtration by $A_*$-subcomodules, i.e., 
\[ \psi_{\ell,B}(L_\ell)\subset A_*\otimes_{\mathbb{F}_p} L_{\ell}\]
for each $\ell$. Consequently, nonzero elements in distinct $t$-weight filtrations in $L$ cannot cancel each other. 

We now consider the coaction on $\sum_{i}x_i\cdot t^{i}$ where we assume that all $x_i$ are nonzero and the sum is therefore a finite sum. 
Then 
\begin{align*} \psi_{\ell,B}(\sum_{i}t^{i}\cdot x_i) &=\sum_i(1\otimes x_i+ \sum_{j\ge 0}y_j^i\otimes z_j^i)(\sum_{j\ge 0}\xi_j\otimes t^{p^j})^i \end{align*}
in which the sum is taken in the graded $\mathbb{F}_p$-algebra $\mathbb{F}_p[t^{\pm 1}]\otimes_{\mathbb{F}_p} B$, and then terms with $t$-weight $>k$ are omitted.
As nonzero terms of distinct $t$-weight cannot cancel each other, it suffices to show that the reduction 
\[\overline{\psi}_{k,B}(\sum_{i}x_i\cdot t^{i})\subset A_*\otimes_{\mathbb{F}_p} L_i/L_{i+1}\]
of $\psi_{k,B}(\sum_{i}x_i\cdot t^{i})$ is nontrivial after subtracting $1\otimes \sum_{i}x_i\cdot t^i$.
We therefore need to show that 
\begin{align*}\overline{\psi}_{k,B}(\sum_{i}x_i\cdot t^{i})-1\otimes \sum_{i}x_i\cdot t^{i} &\ne 0\in A_*\otimes_{\mathbb{F}_p} L_i/L_{i-1}.\end{align*}
By inspection, 
\begin{align*}\overline{\psi}_{k,B}(\sum_{i}x_i\cdot t^{i})-1\otimes \sum_{i}x_i\cdot t^{i} &=( \sum_{j\ge 0}y_j^i\otimes t^{j}\cdot z_j^i)=(1\otimes t^j) \cdot (\sum_{j\ge 0}y_j^i\otimes z_j) ,\end{align*}
but we know that $(\sum_{j\ge 0}y_j^{i}\otimes z_j)\ne 0$, since $B$ has no nonzero $A_*$-comodule primitives in positive degrees. We also know that $z\cdot t^j\ne 0$ when $z\ne 0\in B$ and $j\in \mathbb{Z}$, so this sum must be nontrivial. The conclusion is that $\sum_{i\ge 0}x_i\cdot t^i$
cannot be a comodule primitive, as desired. 
\end{proof}

\begin{lemma}\label{hyp2}
Let $B\langle n\rangle $ be a $p$-primary $E_{2}$ form of $BP\langle n \rangle$ satisfying Running Assumption \ref{assumption}. Suppose $p$ is odd.
For each $n$, there exists a positive integer $M(n)$, independent of $k$, such that the spectrum $TP(B\langle n\rangle)[k]$ satisfies condition $H(M(n))$ for every $k\ge 1$. 
\end{lemma}
\begin{proof}
To prove the lemma, it suffices to check that the graded $\mathbb{F}_p$-vector space of $A_*$-comodule primitives in $H_*(TP(BP\langle n \rangle )[k];\mathbb{F}_p)$ is bounded above. Recall from \eqref{truncated Einf} that the $E_{\infty}$-term of the approximate Tate spectral sequence \eqref{tTSS} is 
\begin{align*} 
\widehat{E}_\infty^{*,*}(BP\langle n \rangle)[k]\cong & \left(P(t^{-1})\{ t^{k-1}\}\otimes_{\mathbb{F}_p} M_1 \otimes_{\mathbb{F}_p} M_2 \right) \oplus  V_k(n)\{t^k\},
\end{align*}
and it strongly converges to $H_*(TP(BP\langle n \rangle )[k];\mathbb{F}_p)$. The graded $A_*$-comodules $M_1$, $M_2$ and $V_k(n)$ were defined in formulas \eqref{M1}, \eqref{M2} and \eqref{vnk} respectively.  
We first claim that that the $A_{*}$-comodule primitives in this $E_{\infty}$-term are bounded above by the integer $M(n)=\sum_{j=1}^{n}|\bxi_j^{p-1}\sigma \bxi_j|$. To prove this it suffices to prove that the graded $A_{*}$-subcomodule of $A_{*}$-comodule primitives in $P(t^{-1})[t^{k-1}]\otimes_{\mathbb{F}_p} M_1 \otimes_{\mathbb{F}_p} M_2$ and $V_k(n)\{t^k\}$ are each trivial in grading degrees greater than $M(n)$. 

We have an isomorphism of graded $A_*$-comodules $M_1\otimes_{\mathbb{F}_p} M_2\cong B\otimes_{\mathbb{F}_p} E$, where $B$ and $E$ are as follows:

$B$ is \[ E(\tau_{j+1}^{\prime}| j\ge n+1 )\otimes P(\bxi_{j+1}|k\ge n+1)\otimes P(\bxi_j^p|1\le j \le n+1),\]
regarded as a graded $A_*$-subcomodule of the dual Steenrod algebra by the map $B\rightarrow A_*$ sending $\tau_{j+1}^{\prime}$ to $\tau_{j+1}$ (this is a consequence of \eqref{nu on tau prime}), sending $\bxi_{j+1}$ to $\bxi_{j+1}$, and sending $\bxi_j^p$ to $\bxi_j^p$. 

$E$ is $E(\bxi_j^{p-1}\sigma \bxi_j |1\le j \le n)$ with trivial $A_*$-coaction, i.e., every element of $E$ is a $A_*$-comodule primitive. This is a consequence of the calculation of the homology of $THH(BP\langle n\rangle)$ in \cite{AR05} together with the calculation of the homological approximate Tate spectral sequence for $THH(BP\langle n\rangle)$ in \cite{BR05}: one just needs to keep track of the $A_*$-coaction while running the spectral sequence.
We give a bit of detail of how this works, by explaining how to verify that $\bxi_1^{p-1}\sigma \bxi_1$ is an $A_*$-comodule primitive; the same verification for the other classes in $E$ is left to the interested reader as an exercise. 

As shown in \cite{BR05}, the homological approximate Tate spectral sequence for $BP\langle n\rangle$ collapses at the $E_{3}$-page. Consequently $\widehat{E}_\infty^{*,*}(BP\langle n \rangle)[k]$ is a subquotient of $E_{2}$. 
The $A_*$-coaction on $H_*(TP(BP\langle n\rangle)[k];\mathbb{F}_p)$
 is computed by first restricting the coaction on the $E_{2}$-page to the $A_{*}$-subcomodule consisting of the $d_{2}$-cycles, and then reducing modulo $d_{2}$-boundaries. 
On the $E_{2}$-term we have 
\[\psi(\bxi_1^{p-1}\sigma \bxi_1) = (\bxi_1\otimes 1 + 1\otimes \bxi_1)^{p-1}(1\otimes \sigma \bxi_1)=\sum_{i=0}^{p-1}\binom{p-1}{i} \bxi_1^i \otimes \bxi_1^{p-1-i} \sigma \bxi_1.\] 
and this remains the coaction after restricting to $d_{2}$-cycles. Every summand except $1\otimes \bxi_1^{p-1} \sigma \bxi_1$ is a $d_2$-boundary. Consequently $\bxi_1^{p-1}\sigma \bxi_1$ is an $A_*$-comodule primitive.

It therefore follows that the $A_*$-subcomodule of primitives in 
\[ P(t^{-1})[t^{k-1}]\otimes_{\mathbb{F}_p} M_1 \otimes_{\mathbb{F}_p} M_2 \]
is trivial in grading degrees greater than $M(n)=\sum_{j=1}^{n}|\bxi_j^{p-1}\sigma \bxi_j|$ by Lemma \ref{comodule primitive lemma}.

From Theorem 5.12 of \cite{AR05}, we know that
\begin{align*} H_*(THH(BP\langle n \rangle);\mathbb{F}_p) &\cong E(\sigma\overline{\xi}_1,\dots ,\sigma\overline{\xi}_{n+1})\otimes_{\mathbb{F}_p} P(\sigma\overline{\tau}_{n+1})\otimes_{\mathbb{F}_p} H_*(BP\langle n\rangle ;\mathbb{F}_p)
\end{align*}
and by formulas \eqref{tau coprod} and \eqref{nu sigma} 
the $A_{*}$-comodule primitives in $H_*(THH(BP\langle n \rangle);\mathbb{F}_p)\{t^{k}\}$ are contained in $\ker d_{2}^{-2k,*}$, where 
\[d_{2}^{-2k,*}\colon \thinspace H_{*}(THH(BP\langle n\rangle;\mathbb{F}_{p})\{t^{k}\}\to H_{*}(THH(BP\langle n\rangle;\mathbb{F}_{p})\{t^{k+1}\}\]
is the $d_{2}$-differential in the approximate Tate spectral sequence. Consequently, the map 
\[ \Hom_{A_{*}}(\mathbb{F}_{p},\ker (d_{2}^{-2k,*}))\to \Hom_{A_{*}}(\mathbb{F}_{p},H_{*}(THH(BP\langle n\rangle);\mathbb{F}_{p}))\]
induced by the canonical inclusion is an epimorphism. We then consider the map of long exact sequences
\[ 
	\xymatrix{
	\Hom_{A_{*}}(\mathbb{F}_{p},\ker (d_{2}^{-2k,*}))\ar@{->>}[r]\ar[d] & \Hom_{A_{*}}(\mathbb{F}_{p},H_{*}(THH(BP\langle n\rangle);\mathbb{F}_{p})) \ar[d] \\ 
		\Hom_{A_{*}}(\mathbb{F}_{p},\ker (d_{2}^{-2k,*})/\im(d_{2}^{2-2k,*}))\ar[r] \ar[d] & \Hom_{A_{*}}(\mathbb{F}_{p},H_{*}(THH(BP\langle n\rangle);\mathbb{F}_{p})/\im(d_{2}^{2-2k,*})) \ar[d] \\ 
	\Ext_{A_{*}}^{1}(\mathbb{F}_{p}, \im(d_{2}^{2-2k,*})\ar[r]^{\id} \ar[d] & \Ext_{A_{*}}^{1}(\mathbb{F}_{p}, \im(d_{2}^{2-2k,*}) \ar[d] \\
	\Ext_{A_{*}}^{1}(\mathbb{F}_{p},\ker (d_{2}^{-2k,*}))\ar@{^{(}->}[r] & \Ext_{A_{*}}^{1}(\mathbb{F}_{p},H_{*}(THH(BP\langle n\rangle);\mathbb{F}_{p}))
	}
\] 
and observe that 
\[ V_k(n)\{t^{k}\}\cong \left ( H_*(THH(BP\langle n \rangle);\mathbb{F}_p)\{t^{k}\} \right )/\im (d_2^{2-2k,*}).\]
We already checked that the $A_{*}$-comodules in $\ker d^{-2k,*}/\im d^{2-2k,*}=M_1 \otimes_{\mathbb{F}_p} M_2$ are bounded above by $M(n)$, so we conclude that the set of $A_{*}$-comodule primitives in $V_{k}(n)$ is also bounded above by $M(n)$ by the $4$-lemma. 

In sum, this proves that the associated graded of a filtration\footnote{Namely, the filtration on the abutment of the approximate Tate spectral sequence \eqref{tTSS}, whose associated graded is $\widehat{E}_{\infty}^{*,*}(R)[k]$.} on 
\[ H_*(TP(BP\langle n\rangle )[k];\mathbb{F}_p)\]
has comodule primitives bounded above by $M(n)$. Because we are working with the {\em approximate} Tate spectral sequence, this filtration is finite in each grading degree: it is of the form 
\[ H_*(TP(BP\langle n\rangle )[k];\mathbb{F}_p) \supset   H_*(TP(BP\langle n\rangle )[k-1];\mathbb{F}_p)\supset \dots \supset  0\]
where $H_s(TP(BP\langle n\rangle )[k];\mathbb{F}_p)=0$ for $k>-s/2$ so in a fixed grading degree $s$ it is the filtration 
\[ H_s(TP(BP\langle n\rangle )[k];\mathbb{F}_p) \supset \dots \supset  H_{s}(TP(BP\langle n\rangle )[-\lceil s/2 \rceil ];\mathbb{F}_p)\supset 0 = 0 = \dots \]
which is a finite filtration, where $\lceil s/2 \rceil$ is the integer ceiling of $s/2$. 

Hence, if there were a nontrivial comodule primitive 
$z\in F_i$ in internal degree $>M(n)$, then either it would map to a comodule primitive in $F_{i}/F_{i-1}$, or it would pull back to an element in $F_{i-1}$. By a finite downward induction on filtration degree, $z$ must be a comodule primitive in $F_{j}/F_{j-1}$ for some $\ell <j<i$. Therefore, it suffices to check that there are no comodule primitives in internal degrees $>M(n)$ in the associated graded of this filtration on $H_*(TP(R)[k];\mathbb{F}_p)$, which we have already done. 
\end{proof} 

Lemma \ref{acyclicization lemma} must surely be known in some form or another: it is quite close to well-developed ideas dating back to Bousfield and Kan's book \cite{MR0365573}. We provide a proof since we have not been able to locate the result in the existing literature. We will use the following notation: given a prime number $p$ and a spectrum $X$, we write $\gamma_pX$ for the homotopy fiber of the $p$-completion map $X \rightarrow X^{\wedge}_p$. 
\begin{lemma}\label{acyclicization lemma}
Let $p$ be a prime number, and let $\{Y_{i}\}$ be a sequence of $p$-local spectra. Suppose that, for each integer $n$, the abelian group $\pi_*(Y_n)$ is $p$-reduced. Suppose furthermore that the derived limit $R^1\lim_n \pi_*(Y_n)$ is trivial. Then the homotopy limit $\holim_i \gamma_p(Y_i)$ is $S/p$-acyclic, i.e., $\left(\holim_i \gamma_p(Y_i)\right)_p^{\wedge}$ is contractible.
\end{lemma}
\begin{proof}
Write $\Lambda: \Mod(\mathbb{Z}_{(p)}) \rightarrow\Mod(\mathbb{Z}_{(p)})$ for the $p$-adic completion functor, i.e., $\Lambda M = \lim_{n} M/p^nM$. It is well-known that $\Lambda$ is neither left nor right exact, and consequently its $0$th left-derived functor $L_0\Lambda$ may fail to coincide with $\Lambda$ itself; see \cite{MR2855123} or the appendix of \cite{MR1388895} for surveys of these ideas. By a special case of Harrison duality\footnote{Also a special case of dualities that generalize Harrison duality, e.g. Matlis duality \cite{MR0178025} and Greenlees-May duality \cite{MR1172439}.} \cite{MR0104728}, the group $L_n\Lambda M$ is naturally isomorphic to $\Ext^{1-n}_{\mathbb{Z}}\left(\mathbb{Z}/p^{\infty},M\right)$, hence vanishes if $n>1$.

For each integer $j$, we have a commutative diagram with exact columns
\[\xymatrix{ 
 \hom_{\mathbb{Z}_{(p)}}\left(\mathbb{Q},\pi_j(X)\right) \ar[d] & \\
  \hom_{\mathbb{Z}_{(p)}}\left(\mathbb{Z}_{(p)},\pi_j(X)\right) \ar[r]^(.6){\cong} \ar[d] &\pi_j(X) \ar[d] \\
  \Ext^1_{\mathbb{Z}_{(p)}}\left(\mathbb{Z}/p^{\infty},\pi_j(X)\right) \ar[r]^(.6){\cong}\ar[d] & L_0\Lambda\pi_j(X) \\
  \Ext^1_{\mathbb{Z}_{(p)}}\left(\mathbb{Q},\pi_j(X)\right) \ar[d] &\\
  0 & 
}\]
induced by the short exact sequence of abelian groups
\[ 0 \rightarrow \mathbb{Z}_{(p)}\rightarrow \mathbb{Q}\rightarrow \mathbb{Z}/p^{\infty}\rightarrow 0 .\]
The group $\Ext^1_{\mathbb{Z}_{(p)}}\left(\mathbb{Q},M\right)$ is uniquely $p$-divisible\footnote{This is of course classical. An easy way to see that it is true: first, note that $\hom_{\mathbb{Z}_{(p)}}(U,D)$ is uniquely $p$-divisible whenever $D$ is $p$-divisible and $U$ is uniquely $p$-divisible. Now embed $M$ into an injective $\mathbb{Z}_{(p)}$-module $I$, so that $I$ and $I/M$ are both $p$-divisible. Then $\Ext^1_{\mathbb{Z}_{(p)}}(\mathbb{Q},M)$ is the cokernel of a homomorphism $\hom_{\mathbb{Z}_{(p)}}(\mathbb{Q},I) \rightarrow \hom_{\mathbb{Z}_{(p)}}(\mathbb{Q},I/M)$ whose domain and codomain are both uniquely $p$-divisible.} for any $\mathbb{Z}_{(p)}$-module $M$. 
Consequently the cokernel of the derived completion map 
\begin{align}\label{derived cpltn map 1}
 \pi_j(X) &\rightarrow L_0\Lambda\pi_j(X)
\end{align}
is uniquely $p$-divisible. Since $\pi_j(X)$ was assumed to be $p$-reduced, the kernel of \eqref{derived cpltn map 1} is trivial.
 
It is classical (see Proposition 2.5 in \cite{MR551009}) that, for any spectrum $X$, we have a short exact sequence
\[ 0 
  \rightarrow L_0\Lambda \pi_j(X)
  \rightarrow \pi_j\left( X_p^{\wedge}\right) 
  \rightarrow L_1\Lambda\pi_{j-1}(X)
  \rightarrow 0\]
for each integer $j$. For each $i$, the group $L_1\Lambda\pi_*Y_i\cong \hom_{\mathbb{Z}}\left(\mathbb{Z}/p^{\infty},\pi_*(Y_i)\right)$ is trivial since $\pi_*(Y_i)$ was assumed to be $p$-reduced. 

Hence $\pi_*(\gamma_pY_i)$ is uniquely $p$-divisible for all $i$.
A limit and derived limit of uniquely $p$-divisible $\mathbb{Z}_{(p)}$-modules remains uniquely $p$-divisible, since such a limit and derived limit can simply be calculated in $\mathbb{Q}$-vector spaces. Hence $\holim_i \gamma_pY_i$ is an $H\mathbb{Q}$-module spectrum, hence is $S/p$-acyclic. 
\end{proof}

\begin{lemma}\label{lem: finiteness}
Suppose $R$ is an $\mathbb{E}_{1}$-ring spectrum satisfying 
\[\mathrm{Ext}_{A_{*}}^{*,*}(\mathbb{F}_{p},H_{*}(R))\cong \mathbb{F}_{p}[x_{i} : 1\le i \le n]\] 
with $x_{i}$ in even, nonnegative internal degrees for $i\ge 1$. Then, there exists an integer $L$ such that 
$\pi_{s}\left ( \mathrm{TP}(R)/p \right )$ is finite for all $s\ge L$. 
\end{lemma}
\begin{proof}
Consider the functor $\mathbb{Z}^{\op}\to \mathrm{Sp}$ defined by 
$	 
\mathrm{fil}_{\mathbb{F}_{p}}^{q}(R)=\mathrm{Tot} \left ( \text{Wh}^{q}(R \wedge H\mathbb{F}_{p}^{\wedge \bullet +1})) \right )
$
with associated graded 
$
	\mathrm{gr}_{\mathbb{F}_{p}}^{s}(R)=\mathrm{Tot} H\pi_{s} (R \wedge H\mathbb{F}_{p}^{\wedge \bullet +1}).
$ 
Then as in \cite{AKS18}, we consider the filtered object $ \mathrm{TP}(\mathrm{fil}_{\mathbb{F}_{p}}^{\bullet}(R))/p$, where we take the quotient in filtered spectra with $p$ in filtration $0$. By \cite[Theorem 3.3.10]{AKS18} (cf. \cite[
Corollary 4.14]{Kee25}), we know that there is an $\mathbb{T}$-equivariant equivalence 
\[ \mathrm{gr}^{*}\mathrm{THH}(\mathrm{fil}_{\mathbb{F}_{p}}^{\bullet}(R)) \simeq \mathrm{THH}(\mathrm{gr}_{\mathbb{F}_{p}}^{*}(R)) .\]
Also, the functor $\mathrm{gr}^{*}$ is a left adjoint by \cite[Lemma 3.30]{GP18} so it commutes with homotopy colimits and we have 
\[ 
	\mathrm{gr}^{*} \left ( \mathrm{THH}(\mathrm{fil}_{\mathbb{F}_{p}}^{\bullet}(R))_{h\mathbb{T}} \right ) \simeq \left (  \mathrm{gr}^{*}\mathrm{THH}(\mathrm{fil}_{\mathbb{F}_{p}}^{*}(R))  \right )_{h\mathbb{T}}\simeq \mathrm{THH}(\mathrm{gr}_{\mathbb{F}_{p}}^{*}(R))_{h\mathbb{T}} .
\]
Since $B\mathbb{T}$ has finite skeleta $\mathrm{sk}_{n}B\mathbb{T}$ using the standard simplicial CW filtration on $B\mathbb{T}=\mathbb{C}P^{\infty}$, we know that 
   \[ 
	 \left ( \mathrm{gr}^{*} \lim_{\mathrm{sk}_{n}B\mathbb{T}} \left ( \mathrm{THH}(\mathrm{fil}_{\mathbb{F}_{p}}^{\bullet}(R))\right ) \right ) \simeq \lim_{\mathrm{sk}_{n}B\mathbb{T}} \mathrm{gr}^{*} \left (  \left ( \mathrm{THH}(\mathrm{fil}_{\mathbb{F}_{p}}^{*}(R) ) \right )  \right )\simeq  \lim_{\mathrm{sk}_{n}B\mathbb{T}} \mathrm{THH}(\mathrm{gr}_{\mathbb{F}_{p}}^{*}(R)) .
\]
Finally, the canonical map 
   \[ 
\mathrm{gr}^{s} \left (\left ( \mathrm{THH}(\mathrm{fil}_{\mathbb{F}_{p}}^{\bullet}(R))\right )^{h\mathbb{T}}\right )  \longrightarrow \lim_{n}\left ( \mathrm{gr}^{s}	\lim_{\mathrm{sk}_{n}B\mathbb{T}}  \left (  \mathrm{THH}(\mathrm{fil}_{\mathbb{F}_{p}}^{\bullet}(R)) \right ) \right )
\]
is an equivalence because
\[  \lim_{n}\mathrm{gr}^{s} F_{n}(R) \simeq 0\]
where 
\[ F_{n}\colon = \fib \left ( (\mathrm{THH}(\mathrm{fil}_{\mathbb{F}_{p}}^{\bullet}(R)))^{h\mathbb{T}} \to \lim_{\mathrm{sk}_{n}B\mathbb{T}}  (  \mathrm{THH}(\mathrm{fil}_{\mathbb{F}_{p}}^{\bullet}(R))  ) \right ).\]
We then use the fiber sequence 
\[ \mathrm{THH}(\mathrm{fil}_{\mathbb{F}_{p}}^{q}(R))^{h\mathbb{T}}\to  \mathrm{THH}(\mathrm{fil}_{\mathbb{F}_{p}}^{q}(R))^{t\mathbb{T}}\to  \Sigma^{2}\mathrm{THH}(\mathrm{fil}_{\mathbb{F}_{p}}^{q}(R))_{h\mathbb{T}}\]
to conclude that 
\[ 
	\mathrm{gr}^{*} \left ( \mathrm{THH}(\mathrm{fil}_{\mathbb{F}_{p}}^{\bullet}(R))^{t\mathbb{T}} \right ) \simeq \mathrm{THH}(\mathrm{gr}_{\mathbb{F}_{p}}^{*}(R))^{t\mathbb{T}} .
\]

Finally, there is an equivalence
\[ 
	\mathrm{TP}(\mathrm{gr}_{\mathbb{F}_{p}}^{*}(R))/p\simeq  \gr^{*} \left ( \mathrm{TP}(\mathrm{fil}_{\mathbb{F}_{p}}^{*}(R))/p \right ).
\]
For conditional convergence, we observe that 
\[\lim_{s}\mathrm{TP}(\mathrm{fil}_{\mathbb{F}_{p}}^{s}(R))/p=0. \]
This follows because $\mathrm{THH}(\mathrm{fil}_{\mathbb{F}_{p}}^{s}(R))$ is increasingly connective as $s\rightarrow\infty$, implying that 
\[ \lim_{s}  ( \mathrm{THH}(\mathrm{fil}_{\mathbb{F}_{p}}^{s}(R))^{h\mathbb{T}}) = ( \lim_{s}   \mathrm{THH}(\mathrm{fil}_{\mathbb{F}_{p}}^{s}(R))^{h\mathbb{T}})=0\]
and 
\[ \lim_{s}  ( \mathrm{THH}(\mathrm{fil}_{\mathbb{F}_{p}}^{s}(R))_{h\mathbb{T}}) = ( \lim_{s}   \mathrm{THH}(\mathrm{fil}_{\mathbb{F}_{p}}^{s}(R))_{h\mathbb{T}})=0\]
where we also use the fact that $B\mathbb{T}$ has finite skeleta for the last identification. 
It is also clear that 
\[ \underset{s}{\colim}\mathrm{TP}(\mathrm{fil}_{\mathbb{F}_{p}}^{s}(R))/p=\mathrm{TP}(R)/p\]
for similar reasons, so the associated spectral sequence 
\[ E_{1}^{s,2t-s}=\pi_{s}([\mathrm{TP}(\mathrm{gr}_{\mathbb{F}_{p}}^{*}(R))/p]^{t})\implies \pi_{s}\mathrm{TP}(R)/p\]
conditionally converges in the sense of \cite{MR1718076}. 

By \cite[Lemma IV.4.12]{NS18}, there is an equivalence
\[\mathrm{TP}(\mathrm{gr}_{\mathbb{F}_{p}}^{*}(R))/p\simeq \mathrm{THH}(\mathrm{gr}_{\mathbb{F}_{p}}^{*}(R))^{tC_{p}}\]
and by \cite[Proposition 4.2.2]{HW22} we know that the Frobenius map
\[ \varphi_{p}\colon  \mathrm{THH}(\mathrm{gr}_{\mathbb{F}_{p}}^{*}(R))\to \mathrm{THH}(\mathrm{gr}_{\mathbb{F}_{p}}^{*}(R))^{tC_{p}}\]
is an isomorphism on $\pi_{s}$ for all $s\ge L$ for some integer $L$. It is clear that $\pi_{s}\mathrm{THH}(\mathrm{gr}_{\mathbb{F}_{p}}^{*}(R))$
is a finitely type graded $\mathbb{F}_{p}$-vector space for each $s$, by two standard K\"unneth spectral sequence arguments. Consequently, we conclude that the $\mathbb{F}_{p}$-vector space $\pi_{s}(\mathrm{TP}(\mathrm{gr}_{\mathbb{F}_{p}}^{*}(R))/p)$ is finite for each $s\ge L$. Finally, we observe that we can apply the Beilinson $t$-structure \cite[II.2.1]{Hed20} to produce a filtered object
\[ \tau_{\ge L}^{\textup{Bei}}\mathrm{TP}(\mathrm{fil}_{\mathbb{F}_{p}}^{s}(R)) \]
whose associated spectral sequence has $E_{1}$-term 
\[ E_{1}^{s,*}=\begin{cases} 
\mathrm{THH}(\mathrm{gr}_{\mathbb{F}_{p}}^{*}(R) )^{tC_{p}}& \text{ if } s\ge L \\
0 &  \text{otherwise}
\end{cases}
\] 
and converges to $\pi_{*}\Wh^{L}\mathrm{TP}(R)/p$. 
Consequently, this spectral sequence strongly converges and we conclude that the groups $\pi_{s}(\mathrm{TP}(R)/p)$ are finite for all $s\ge L$ as desired. 
\end{proof}

\begin{theorem} \label{main application}
Let $m,n$ be nonnegative integers such that $m\geq n+2$. Let $R$ be a $p$-primary $E_{2}$ form of $BP\langle n\rangle$ satisfying Running Assumption \ref{assumption}. Suppose $p$ is odd. Then there are isomorphisms
\begin{align}
\nonumber K(m)_*(TP(R))\cong &0 \text{ and } \\
\nonumber K(m)_*(TC^{-}(R))\cong & 0 .
\end{align} 
\end{theorem}
\begin{proof}
We claim that the sequence of spectra
\[ \dots \rightarrow TP(R)[1]^{\wedge}_p \rightarrow TP(R)[0]^{\wedge}_p \rightarrow TP(R)[-1]^{\wedge}_p \rightarrow \dots\]
is $K(m)$-amenable.
We check the conditions for $K(m)$-amenability, given in Definition \ref{def of Km-amenability}, as follows:
\begin{enumerate}
\item It is straightforward from the definitions that $TP(R)[k]^{\wedge}_p$ is $p$-complete and bounded below.
\item The homology $H_*(TP(R)[k]^{\wedge}_p;\mathbb{F}_p)$ is finite-type, for all $k$. 
\item The pro-triviality of the sequence $\{H(H_*(TP(R)[\bullet]^{\wedge}_p ; \mathbb{F}_p),Q_{n})\}$ is Proposition \ref{Margoliscomputaton}.
\item There exists an integer $L$ such that the groups $(\pi_{s}TP(R))/p$ are finite for $s\ge L$ by Lemma \ref{lem: finiteness}.
\item There exists an integer $M$ such that the spectrum $TP(R)[k]^{\wedge}_p$ satisfies condition $H(M)$ for all $k$, by Lemma \ref{hyp2}.
\end{enumerate}
Hence $\holim_k \left(TP(R)[k]^{\wedge}_p\right)$ is $K(m)$-acyclic, by Theorem \ref{thm on holim of margolis-acyclics with lim1 vanishing}. 

Consequently have a commutative square of spectra whose rows and columns are homotopy fiber sequences:
\begin{equation}\label{comm diag 20932}\xymatrix{
 \gamma_p\holim_k\gamma_p\left(TP(R)[k]\right) \ar[r]^(.6){\simeq}\ar[d]_{\simeq}
  & \gamma_p TP(R) \ar[d] 
  & \\
 \holim_k\gamma_p\left(TP(R)[k]\right) \ar[r]
  & TP(R) \ar[r]\ar[d]
  & \holim_k\left( TP(R)[k]^{\wedge}_p\right) \ar[d]^{\simeq} \\
  & \left(TP(R)\right)^{\wedge}_p \ar[r]_(.4){\simeq} 
  & \left( \holim_k\left( (TP(R)[k])^{\wedge}_p\right)\right)^{\wedge}_p.
}\end{equation}
Every $S/p$-local equivalence is also a $K(m)$-local equivalence, so the spectra in the top row of \eqref{comm diag 20932} are all $K(m)$-acyclic. We have already shown that $\holim_k \left(TP(R)[k]^{\wedge}_p\right)$ is $K(m)$-acyclic. Consequently $TP(R)$ and $\left(TP(R)\right)^{\wedge}_p$ are each $K(m)$-acyclic.

Since $\holim_k TP(R)[k]\simeq TP(R)\simeq THH(R)^{t\mathbb{T}}$, the vanishing of 
\[K(m)_*\holim_k TP(R)[k]\]for $m\geq n+2$ together with the fiber sequence
\[ \Sigma THH(R )_{h\mathbb{T}} \to THH(R )^{h\mathbb{T}}\to THH(R )^{t\mathbb{T}} \]
implies that the map $\Sigma THH(R )_{h\mathbb{T}} \to THH(R )^{h\mathbb{T}}$ is a $K(m)$-local equivalence for $m\geq n+2$.

Consequently, if we prove that $THH(R )_{h\mathbb{T}}$ is also $K(m)$-acyclic, then $TC^-(R) \simeq THH(R )^{h\mathbb{T}}$ must also be $K(m)$-acyclic. This, however, is quite straightforward: when $m\ge n+1$, the vanishing of $K(m)_*(BP\langle n \rangle)$ was proven in \cite[Thm. 2.1(d),(f),(i)]{MR737778}. Hence $THH(R)$ is an algebra over a $K(m)$-acyclic ring spectrum (namely, $R$), hence $THH(R)$ is $K(m)$-acyclic. Smashing with $K(m)$ commutes with homotopy colimits, so
\begin{align*}
 K(m)\smash (THH(R)_{h\mathbb{T}})
  &\simeq  \left(K(m)\smash THH(R)\right)_{h\mathbb{T}} \simeq 0\end{align*}
for $m\ge n+1$.
\end{proof}
\begin{corollary} 
Suppose $p$ is odd. If $B\langle n\rangle $ is a $p$-primary $E_{2}$ form of $BP\langle n\rangle$ satisfying Running Assumption \ref{assumption}, then there are weak equivalences 
\[ L_{K(m)}K(B\langle n\rangle  )\simeq 0 \]
for $m\ge n+2\ge 2$.
\end{corollary} 
\begin{proof}
We first show that there is an equivalence 
\[ L_{K(m)}TC(B\langle n\rangle )\simeq 0\]
for $m\ge n+2\ge 2$. This follows by Theorem \ref{main application} together with the long exact sequence in Morava $K$-theory associated to the homotopy  fiber sequence 
\[ TC(B\langle n\rangle )^{\wedge}_p \to TC^{-}(B\langle n\rangle )_p^{\wedge}\overset{\text{can}-\varphi_{p}^{h\mathbb{T}}}{\longrightarrow} TP(B\langle n\rangle)_p^{\wedge}\]
of \cite{NS18}. 
The fact that 
\[ L_{K(m)}K(\mathbb{Z}_{(p)})=0\]
for $m\ge 2$ then follows by \cite[Theorem D]{HM97}. 
For $n>0$, the result follows from Theorem \ref{main application} by the Dundas--Goodwillie--McCarthy theorem \cite[Theorem 7.2.2.1]{MR3013261} and the $n=0$ case, which together imply that \[ L_{K(m)}K(BP\langle n\rangle ) \simeq L_{K(m)}TC(BP\langle n\rangle)\]
for $m\ge n+2\ge 3$. 
\end{proof}

\appendix

\section{Brief review of Margolis homology.}\label{appendix on basics of margolis homology}
This appendix, which does not logically rely on anything earlier in the paper, consists of material that is well-known to users of Margolis homology. This material is not difficult and certainly not new. Nevertheless we include it in this paper because, for some of this material, we do not know of a clear and straightforward account in the existing literature.

Given a graded ring $R$, we write $\gr\Mod(R)$ for the category of graded $R$-modules and grading-preserving $R$-module homomorphisms.
\begin{definition}\label{def of margolis homology}
Let $k$ be a field and let $E(Q)$ be the exterior $k$-algebra on a single homogeneous generator $Q$ in an odd grading degree $\left|Q\right|$.
By {\em Margolis $Q$-homology} we mean the functor $H(-;Q)\colon\thinspace \gr\Mod(E(Q)) \rightarrow\gr\Mod(E(Q))$ given on a graded $E(Q)$-module $M$ by the quotient
\[ H(M; Q) = \left( \ker (M \stackrel{Q}{\longrightarrow} \Sigma^{-\left|Q\right|}M) \right)/\left( \im (\Sigma^{\left|Q\right|}M \stackrel{Q}{\longrightarrow} M) \right).\]
\end{definition}
It is routine to verify that $H(-; Q_n)$ sends short exact sequences of graded $E(Q_n)$-modules to long exact sequences.

Here is a quick note on gradings; it is extremely elementary, but not taking a moment to ``fix notations'' on this point tends to lead to sign errors in the gradings.
\begin{convention}\label{convention on grading degrees}
Given a graded ring $R$ and graded $R$-modules $M$ and $N$, we write $\Hom_R(M,N)$ for the degree-preserving $R$-linear morphisms $M\rightarrow N$, and we write $\underline{\Hom}_R(M,N)$ for the graded abelian group whose degree $n$ summand is $\Hom_R(\Sigma^n M,N)$.
We write $\Ext^{s,*}_R(M,N)$ for the graded abelian group whose degree $t$ summand is $\Ext^{s,t}_R(M,N)$, and we refer to this grading as the {\em internal} or {\em topological} grading, to distinguish it from the {\em cohomological} degree given by $s$.
\end{convention}
In particular, the $k$-linear dual of a graded $k$-vector space has the signs of the gradings reversed, i.e., \[\left( \Sigma^n V\right)^* = \underline{\Hom}_k\left(\Sigma^n V,k\right) \cong \Sigma^{-n}\left(V^*\right).\]

Now given a spectrum $X$, the action of $Q_n$ on $H^*(X; \mathbb{F}_p)$ is the one induced in homotopy by the map of function spectra $F(X, H\mathbb{F}_p) \rightarrow F(X, \Sigma^{2p^n-1}H\mathbb{F}_p)$ induced by the composite~\eqref{Qn map of spectra}. Somewhat less famous than the action of Steenrod operations on mod $p$ cohomology, we have also the dual action of Steenrod operations on mod $p$ homology: the action of $Q_n$ on $H_*(X; \mathbb{F}_p)$ is the one induced in homotopy by the map of spectra $X\smash H\mathbb{F}_p\rightarrow X\smash \Sigma^{2p^n-1}H\mathbb{F}_p$ induced by the composite~\eqref{Qn map of spectra}. 
These operations are $\mathbb{F}_p$-linearly dual under the isomorphism $H^i(X; \mathbb{F}_p) \cong \Hom_{\mathbb{F}_p}(H_i(X; \mathbb{F}_p), \mathbb{F}_p)$; 
see Proposition~III.13.5 of~\cite{MR1324104} or Theorem~IV.4.5 of~\cite{MR1417719}. 

\begin{lemma}\label{dualization iso}
Let $k,E(Q),\left|Q\right|$ be as in Definition \ref{def of margolis homology}. Let $M$ be a graded $E(Q)$-module. Then we have an isomorphism of graded $E(Q)$-modules
\[ \underline{\Hom}_{E(Q)}(M,E(Q))\cong \underline{\Hom}_{k}(M,\Sigma^{\left|Q\right|}k) \]
natural in the choice of $M$.
\end{lemma}
\begin{proof}
Since each of the functors $\underline{\Hom}_{E(Q)}(-,E(Q))$ and $\underline{\Hom}_{k}(-,\Sigma^{\left|Q\right|}k)$ send colimits to limits, it suffices to prove the isomorphism for finitely generated $E(Q)$-modules. Since $E(Q)$ is a PID, all finitely generated $E(Q)$-modules are direct sums of suspensions of $E(Q)$ and $k$. Therefore, it suffices to prove the statement for suspensions of $E(Q)$ and of $k$, but in this case it is clear that the statement holds.
\end{proof}
\begin{prop}\label{margolis homology and ext}
Let $k,E(Q),\left|Q\right|$ be as in Definition \ref{def of margolis homology}. 
Then we have natural isomorphisms of graded $E(Q)$-modules:
\begin{align}
 \label{iso 204947} H(M^*; Q) &\cong H(M; Q)^* , \\ 
 \label{iso 204948} \Ext^{s,*}_{E(Q)}(k,M) &\cong \Sigma^{-s\left|Q\right|} H(M; Q) \mbox{\ \ if\ } s>0, \\ 
 \label{iso 204949} \Ext^{s,*}_{E(Q)}(M,k) &\cong \Sigma^{-(s+1)\left|Q\right|} H\left(\underline{\Hom}_{E(Q)}(M,E(Q));Q\right) \mbox{\ \ if\ } s>0, \\ 
 \label{iso 204950}                      &\cong \Sigma^{-s\left|Q\right|} \left( H\left(M;Q\right)^*\right) \mbox{\ \ if\ } s>0.
\end{align}
natural in the choice of graded $E(Q)$-module $M$. (The notation $M^*$ is for the graded $k$-linear dual of $M$, i.e., $M^* = \underline{\Hom}_k(M,k)$.)
\end{prop}
\begin{proof}
We handle each of the isomorphisms \eqref{iso 204947} through \eqref{iso 204950} in turn:

 The $Q$-Margolis homology of $M$ is the homology of the chain complex
\begin{align}\label{cc 1130} \dots \stackrel{Q}{\longrightarrow} \Sigma^{\left|Q\right|}M \stackrel{Q}{\longrightarrow} M \stackrel{Q}{\longrightarrow} \Sigma^{-\left|Q\right|}M \stackrel{Q}{\longrightarrow} \dots\end{align}
and so, since the $k$-linear dual of the multiplication-by-$Q$ map on a $E(Q)$-module is the multiplication-by-$Q$ map on the $k$-linear dual of that module, the cohomology of the $k$-linear dual of the chain complex \eqref{cc 1130} is $H(M^*; Q)$. So the classical universal coefficient sequence for chain complexes (e.g. as in 3.6.5 of \cite{MR1269324}) yields the isomorphism \eqref{iso 204947}.

Applying $\underline{\Hom}_{E(Q)}( -,M)$ to the projective graded $E(Q)$-module resolution of $k$
\begin{align}\label{proj res 0394} 0 \leftarrow E(Q) \stackrel{Q}{\longleftarrow} \Sigma^{\left| Q\right|} E(Q) \stackrel{Q}{\longleftarrow} \Sigma^{2\left| Q\right|} E(Q) \stackrel{Q}{\longleftarrow} \dots\end{align}
yields the cochain complex 
\[ 0 \rightarrow M \stackrel{Q}{\longrightarrow} \Sigma^{-\left| Q\right|} M \stackrel{Q}{\longrightarrow} \Sigma^{-2\left| Q\right|} M \stackrel{Q}{\longrightarrow} \dots\]
whose homology is $\Sigma^{-s\left| Q\right|} H(M; Q)$ in each cohomological degree $s>0$. This gives us isomorphism \eqref{iso 204948}. (See Convention \ref{convention on grading degrees} for the sign change in grading degrees.)

We take advantage of the fact that $E(Q)$ is self-injective, so that
\begin{align}\label{res 0394} 0 \rightarrow \Sigma^{-\left| Q\right|}E(Q) \stackrel{Q}{\longrightarrow} \Sigma^{-2\left| Q\right|} E(Q) \stackrel{Q}{\longrightarrow} \Sigma^{-3\left| Q\right|} E(Q) \stackrel{Q}{\longrightarrow} \dots\end{align}
is an injective graded $E(Q)$-module resolution of $k$. 
Applying $\underline{\Hom}_{E(Q)}\left( M,-\right)$ to \eqref{res 0394} yields the cochain complex
\begin{align}\label{res 0394a} 0 \rightarrow \underline{\Hom}_{E(Q)}(M,\Sigma^{-\left| Q\right|}E(Q)) \stackrel{Q}{\longrightarrow} \underline{\Hom}_{E(Q)}(M,\Sigma^{-2\left| Q\right|} E(Q)) \stackrel{Q}{\longrightarrow} 
\dots,\end{align}
hence isomorphism \eqref{iso 204949}.

Isomorphism \eqref{iso 204950} then follows from the chain of isomorphisms
\begin{align*}
 \Sigma^{-(s+1)\left|Q\right|} H\left(\underline{\Hom}_{E(Q)}(M,E(Q));Q\right)
  &\cong \Sigma^{-(s+1)\left|Q\right|} H\left(\underline{\Hom}_{k}(M,\Sigma^{\left|Q\right|}k);Q\right) \\
  &\cong \Sigma^{-(s+1)\left|Q\right|} H\left(\Sigma^{\left|Q\right|}M^*;Q\right) \\
  &\cong \Sigma^{-s\left|Q\right|} \left(H\left(M;Q\right)^*\right) ,
\end{align*}
due to Lemma \ref{dualization iso}.
\end{proof} 

Proposition \ref{duality prop} is a simple cohomological duality. For clarity, we drop the gradings:
\begin{prop} \label{duality prop} 
Let $k,E(Q),\left|Q\right|$ be as in Definition \ref{def of margolis homology}.
For each $E(Q)$-module $M$, we have an isomorphism of graded $E(Q)$-modules
\begin{align} \label{iso 15501}  \Ext^{s}_{E(Q)}(M,k) \cong \Ext^{s}_{E(Q)}(k,M^*)\end{align}
for each integer $s$.
If $s>0$, then each side of \eqref{iso 15501} is furthermore isomorphic to
$\Ext^{s,*}_{E(Q)}(k,M)^*.$ 
\end{prop}
\begin{proof}
The $s=0$ case of isomorphism \eqref{iso 15501} is immediate. 
Consequently, for the rest of this proof we assume $s>0$. The hypotheses of Proposition \ref{margolis homology and ext} are then fulfilled.
Stringing together isomorphisms from Proposition \ref{margolis homology and ext}:
\begin{align}
\label{iso 430005050} \Ext^{s}_{E(Q)}(M,k)
                        &\cong \Sigma^{-s\left| Q\right|} \left(H(M; Q)^*\right) \\
\nonumber                        &\cong \Sigma^{-s\left| Q\right|} H(M^*; Q) \\
\nonumber                        &\cong \Ext^{s}_{E(Q)}(k,M^*).
\end{align}
The right-hand side of \eqref{iso 430005050} is also isomorphic to $\Ext^{s}_{E(Q)}(k,M)^*$, by isomorphism \eqref{iso 204948}.
\end{proof}
\begin{corollary}\label{input of Adams for k(n) smash X}
Fix $s\ge 1$ and suppose $H_{*}(X;\mathbb{F}_{p})$ is finite type. Then the graded $\mathbb{F}_p$-vector space $\Ext_{A_{*}}^{s,*}(\mathbb{F}_{p},H_{*}(k(n)\wedge X;\mathbb{F}_{p}))$ is isomorphic to a suspension of $H(H_{*}(X;\mathbb{F}_{p});Q_{n})$.
\end{corollary}
\begin{proof}
Fix $s\ge 1$. We apply the K\"unneth isomorphism and the change-of-rings isomorphism to produce the isomorphism 
\[ \Ext_{A_{*}}^{s,-*}(\mathbb{F}_{p},H_{*}(k(n)\wedge X;\mathbb{F}_{p}))= \Ext_{E(Q_{n})_{*}}^{s,-*}(\mathbb{F}_{p},H_{*}(X;\mathbb{F}_{p}))\]
Since $E(Q_{n})_{*}$ is a self-dual finite-dimensional Hopf algebra and $H_{*}(X;\mathbb{F}_{p})^{*}=H^{*}(X;\mathbb{F}_{p})$, we can apply Proposition \ref{duality prop} to identify $\Ext_{E(Q_{n})_{*}}^{s,-*}(\mathbb{F}_{p},H_{*}(X;\mathbb{F}_{p}))$ with 
$\Ext_{E(Q_{n})}^{s,*}(H^{*}(X;\mathbb{F}_{p}),\mathbb{F}_{p})$. 
Now Proposition \ref{duality prop} finishes the job.
\end{proof}

\bibliography{salch}

\def\cprime{$'$} \def\cprime{$'$} \def\cprime{$'$} \def\cprime{$'$}
\begin{thebibliography}{10}

\bibitem{MR1324104}
J.~F. Adams.
\newblock {\em Stable homotopy and generalised homology}.
\newblock Chicago Lectures in Mathematics. University of Chicago Press,
  Chicago, IL, 1995.
\newblock Reprint of the 1974 original.

\bibitem{AKS18}
Gabe Angelini-Knoll and Andrew Salch.
\newblock A {M}ay-type spectral sequence for higher topological {H}ochschild
  homology.
\newblock {\em Algebr. Geom. Topol.}, 18(5):2593--2660, 2018.

\bibitem{MR4557878}
Gabriel Angelini-Knoll.
\newblock Detecting {$\beta$} elements in iterated algebraic {K}-theory.
\newblock {\em Trans. Amer. Math. Soc.}, 376(4):2657--2692, 2023.

\bibitem{AKQ25}
Gabriel Angelini-Knoll and J.~D. Quigley.
\newblock A homological approach to chromatic complexity of algebraic k-theory.
\newblock To appear in J. Pure Appl. Algebra, 2025.

\bibitem{AL17}
Vigleik Angeltveit and John~A. Lind.
\newblock Uniqueness of {{\(BP \langle {n} \rangle\)}}.
\newblock {\em J. Homotopy Relat. Struct.}, 12(1):17--30, 2017.

\bibitem{AR05}
Vigleik Angeltveit and John Rognes.
\newblock Hopf algebra structure on topological {H}ochschild homology.
\newblock {\em Algebr. Geom. Topol.}, 5:1223--1290, 2005.

\bibitem{MR0245577}
M.~Artin and B.~Mazur.
\newblock {\em Etale homotopy}.
\newblock Lecture Notes in Mathematics, No. 100. Springer-Verlag, Berlin-New
  York, 1969.

\bibitem{AR02}
Christian Ausoni and John Rognes.
\newblock Algebraic {$K$}-theory of topological {$K$}-theory.
\newblock {\em Acta Math.}, 188(1):1--39, 2002.

\bibitem{AR08}
Christian Ausoni and John Rognes.
\newblock The chromatic red-shift in algebraic {K}-theory.
\newblock {\em Enseignement Math{\'e}matique}, 54(2):13--15, 2008.

\bibitem{MR2182697}
Andrew Baker and Birgit Richter.
\newblock On the {$\Gamma$}-cohomology of rings of numerical polynomials and
  {$E_\infty$} structures on {$K$}-theory.
\newblock {\em Comment. Math. Helv.}, 80(4):691--723, 2005.

\bibitem{MR3065177}
Maria Basterra and Michael~A. Mandell.
\newblock The multiplication on {BP}.
\newblock {\em J. Topol.}, 6(2):285--310, 2013.

\bibitem{MR1718076}
J.~Michael Boardman.
\newblock Conditionally convergent spectral sequences.
\newblock In {\em Homotopy invariant algebraic structures ({B}altimore, {MD},
  1998)}, volume 239 of {\em Contemp. Math.}, pages 49--84. Amer. Math. Soc.,
  Providence, RI, 1999.

\bibitem{MR551009}
A.~K. Bousfield.
\newblock The localization of spectra with respect to homology.
\newblock {\em Topology}, 18(4):257--281, 1979.

\bibitem{MR0365573}
A.~K. Bousfield and D.~M. Kan.
\newblock {\em Homotopy limits, completions and localizations}.
\newblock Lecture Notes in Mathematics, Vol. 304. Springer-Verlag, Berlin,
  1972.

\bibitem{BR05}
Robert~R. Bruner and John Rognes.
\newblock Differentials in the homological homotopy fixed point spectral
  sequence.
\newblock {\em Algebr. Geom. Topol.}, 5:653--690, 2005.

\bibitem{CMNN24}
Dustin Clausen, Akhil Mathew, Niko Naumann, and Justin Noel.
\newblock Descent and vanishing in chromatic algebraic {{\(K\)}}-theory via
  group actions.
\newblock {\em Ann. Sci. {\'E}c. Norm. Sup{\'e}r. (4)}, 57(4):1135--1190, 2024.

\bibitem{MR3013261}
Bj{\o}rn~Ian Dundas, Thomas~G. Goodwillie, and Randy McCarthy.
\newblock {\em The local structure of algebraic {K}-theory}, volume~18 of {\em
  Algebra and Applications}.
\newblock Springer-Verlag London, Ltd., London, 2013.

\bibitem{MR1417719}
A.~D. Elmendorf, I.~Kriz, M.~A. Mandell, and J.~P. May.
\newblock {\em Rings, modules, and algebras in stable homotopy theory},
  volume~47 of {\em Mathematical Surveys and Monographs}.
\newblock American Mathematical Society, Providence, RI, 1997.
\newblock With an appendix by M. Cole.

\bibitem{MR908451}
J.~P.~C. Greenlees.
\newblock Representing {T}ate cohomology of {$G$}-spaces.
\newblock {\em Proc. Edinburgh Math. Soc. (2)}, 30(3):435--443, 1987.

\bibitem{MR1172439}
J.~P.~C. Greenlees and J.~P. May.
\newblock Derived functors of {$I$}-adic completion and local homology.
\newblock {\em J. Algebra}, 149(2):438--453, 1992.

\bibitem{GP18}
Owen Gwilliam and Dmitri Pavlov.
\newblock Enhancing the filtered derived category.
\newblock {\em J. Pure Appl. Algebra}, 222(11):3621--3674, 2018.

\bibitem{HW22}
Jeremy Hahn and Dylan Wilson.
\newblock Redshift and multiplication for truncated {Brown}-{Peterson} spectra.
\newblock {\em Ann. Math. (2)}, 196(3):1277--1351, 2022.

\bibitem{MR0104728}
D.~K. Harrison.
\newblock Infinite abelian groups and homological methods.
\newblock {\em Ann. of Math. (2)}, 69:366--391, 1959.

\bibitem{Hed20}
Alice Hedenlund.
\newblock {\em Multiplicative Tate Spectral Sequences}.
\newblock PhD thesis, University of Oslo, 2020.

\bibitem{HM97}
Lars Hesselholt and Ib~Madsen.
\newblock On the {{\(K\)}}-theory of finite algebras over {Witt} vectors of
  perfect fields.
\newblock {\em Topology}, 36(1):29--101, 1997.

\bibitem{MR2671186}
Michael Hill and Tyler Lawson.
\newblock Automorphic forms and cohomology theories on {S}himura curves of
  small discriminant.
\newblock {\em Adv. Math.}, 225(2):1013--1045, 2010.

\bibitem{MR1388895}
Mark Hovey, John~H. Palmieri, and Neil~P. Strickland.
\newblock Axiomatic stable homotopy theory.
\newblock {\em Mem. Amer. Math. Soc.}, 128(610):x+114, 1997.

\bibitem{Jam55}
I.~M. James.
\newblock Reduced product spaces.
\newblock {\em Ann. of Math. (2)}, 62:170--197, 1955.

\bibitem{Kee25}
Liam Keenan.
\newblock The {May} filtration on {THH} and faithfully flat descent.
\newblock {\em J. Pure Appl. Algebra}, 229(1):30, 2025.
\newblock Id/No 107806.

\bibitem{LMMT24}
Markus Land, Akhil Mathew, Lennart Meier, and Georg Tamme.
\newblock Purity in chromatically localized algebraic {{\(K\)}}-theory.
\newblock {\em J. Am. Math. Soc.}, 37(4):1011--1040, 2024.

\bibitem{MR3862946}
Tyler Lawson.
\newblock Secondary power operations and the {B}rown-{P}eterson spectrum at the
  prime 2.
\newblock {\em Ann. of Math. (2)}, 188(2):513--576, 2018.

\bibitem{LN14}
Tyler Lawson and Niko Naumann.
\newblock Strictly commutative realizations of diagrams over the {Steenrod}
  algebra and topological modular forms at the prime 2.
\newblock {\em Int. Math. Res. Not.}, 2014(10):2773--2813, 2014.

\bibitem{MR3007679}
Sverre Lun\o~e Nielsen and John Rognes.
\newblock The topological {S}inger construction.
\newblock {\em Doc. Math.}, 17:861--909, 2012.

\bibitem{MR99}
Mark Mahowald and Charles Rezk.
\newblock Brown-{C}omenetz duality and the {A}dams spectral sequence.
\newblock {\em Amer. J. Math.}, 121(6):1153--1177, 1999.

\bibitem{MR0178025}
Eben Matlis.
\newblock Cotorsion modules.
\newblock {\em Mem. Amer. Math. Soc. No.}, 49:66, 1964.

\bibitem{MR0099653}
John Milnor.
\newblock The {S}teenrod algebra and its dual.
\newblock {\em Ann. of Math. (2)}, 67:150--171, 1958.

\bibitem{MR0174052}
John~W. Milnor and John~C. Moore.
\newblock On the structure of {H}opf algebras.
\newblock {\em Ann. of Math. (2)}, 81:211--264, 1965.

\bibitem{Mit90}
S.~A. Mitchell.
\newblock The {M}orava {$K$}-theory of algebraic {$K$}-theory spectra.
\newblock {\em $K$-Theory}, 3(6):607--626, 1990.

\bibitem{NS18}
Thomas Nikolaus and Peter Scholze.
\newblock On topological cyclic homology.
\newblock {\em Acta Math.}, 221(2):203--409, 2018.

\bibitem{Pet20}
Eric Peterson.
\newblock Coalgebraic formal curve spectra and spectral jet spaces.
\newblock {\em Geom. Topol.}, 24(1):1--47, 2020.

\bibitem{MR0315016}
Daniel Quillen.
\newblock On the cohomology and {$K$}-theory of the general linear groups over
  a finite field.
\newblock {\em Ann. of Math. (2)}, 96:552--586, 1972.

\bibitem{MR737778}
Douglas~C. Ravenel.
\newblock Localization with respect to certain periodic homology theories.
\newblock {\em Amer. J. Math.}, 106(2):351--414, 1984.

\bibitem{Sad2}
Hal Sadofsky.
\newblock The homology of inverse limits of spectra.
\newblock {\em unpublished}.

\bibitem{Sad1}
Hal Sadofsky.
\newblock Morava k-theory of homotopy inverse limits of spectra.
\newblock {\em unpublished}.

\bibitem{Sen24}
Andrew Senger.
\newblock The {Brown}-{Peterson} spectrum is not {{\(\mathbb{E}_{2 (p^2 +
  2)}\)}} at odd primes.
\newblock {\em Adv. Math.}, 458:33, 2024.
\newblock Id/No 109996.

\bibitem{MR1269324}
Charles~A. Weibel.
\newblock {\em An introduction to homological algebra}, volume~38 of {\em
  Cambridge Studies in Advanced Mathematics}.
\newblock Cambridge University Press, Cambridge, 1994.

\bibitem{MR2855123}
Amnon Yekutieli.
\newblock On flatness and completion for infinitely generated modules over
  {N}oetherian rings.
\newblock {\em Comm. Algebra}, 39(11):4221--4245, 2011.

\end{thebibliography}
\bibliographystyle{plain}
\end{document}